\documentclass[preprint]{revtex4-1}

\usepackage{amsthm}
\usepackage[all]{xy}
\usepackage{amsmath}%
\usepackage{amsfonts}%
\usepackage{amssymb}%
\usepackage{graphicx}

\newtheorem{theorem}{Theorem}[section] 
\newtheorem{corollary}{Corollary}[section] 
\newtheorem{definition}{Definition}[section] 
\newtheorem{proposition}{Proposition}[section] 
\newtheorem{remark}{Remark}[section] 
\begin{document}

%\preprint{}

%\title{} %Title of paper
\title{Gabor analysis as contraction of wavelets analysis}

\author{Eyal M. Subag}
\email{eus25@psu.edu}
\affiliation{Department of Mathematics, Pennsylvania State University, State College, Pennsylvania 16802,  USA}
\author{Ehud Moshe Baruch}
\email{embaruch@math.technion.ac.il}
\affiliation{Department of Mathematics, Technion-Israel Institute of  Technology, Haifa 32000, Israel }
\author{Joseph L. Birman}
\thanks{Deceased.}
\affiliation{Department of Physics,
The City College of the City University of   New York, New York N.Y.10031, USA }
\author{Ady Mann}
\email{ady@physics.technion.ac.il}
\affiliation{Department of Physics, Technion-Israel Institute of  Technology, Haifa 32000, Israel}

\begin{abstract}
\noindent We use the method of group contractions to relate wavelets analysis and Gabor analysis. Wavelets analysis is associated with  unitary irreducible representations of the affine group  while Gabor analysis is associated with unitary irreducible representations of the   Heisenberg group.  We obtain unitary irreducible representations of the Heisenberg group as contractions of representations of the extended affine group. Furthermore, we use these contractions to relate the two analyses, namely we contract coherent states, resolutions of the identity, and tight frames. In order to obtain the standard Gabor frame we construct a family of time localized wavelets frames that contract to that Gabor frame. Starting from a standard wavelets frame we construct a family of frequency localized wavelets frames that contract to a nonstandard Gabor frame. In particular we deform Gabor frames to wavelets frames.
\end{abstract}


\maketitle %\maketitle must follow title, authors, abstract and \pacs

\section{Introduction and main results}
\noindent Wavelets and Gabor analyses  are fundamental time-frequency analyses with various applications in quantum physics and signal analysis (e.g. see \cite{ali2013coherent,Chris2016,Feichtinger2003,Feichtinger1998,MR1843717}).  Both are incidents of analysis induced by coherent states system   \cite{perelomov2012generalized,ali2013coherent} that is attached to a unitary irreducible representation of a Lie group. We shall deal with the simplest   case in which wavelet analysis  corresponds to the affine group (also known as the $"ax+b"$ group), and  Gabor analysis corresponds to the Weyl-Heisenberg group. The purpose of this paper is to relate various aspects of the two analyses via  contraction \cite{Inonu-Wigner53,gilmore2012lie,Subag12}  between these groups and their unitary irreducible representations.   

Some relations between Gabor analysis and wavelets analysis were established before in \cite{Tor1,Tor2,Dahlke2008,MR1325539} by exhibiting both the Weyl-Heisenberg group
and the affine group as subgroups of the four dimensional affine Weyl-Heisenberg group. In \cite{Willard},
an approximation of the narrowband  cross-ambiguity function by wideband cross-ambiguity functions is interpreted in terms of
contraction of an extension of the affine group to the Heisenberg group.

Before  stating our results we shall recall the general set up for Perelomov coherent states \cite{perelomov2012generalized,ali2013coherent}.  Let $\pi: G\longrightarrow \mathcal{U}(\mathcal{H})$ be a unitary irreducible
representation  of a Lie group $G$ on a complex Hilbert space $\mathcal{H}$. For any nonzero $\psi\in \mathcal{H}$ let
 $G_{\psi,\pi}$ denote the isotropy subgroup of $\psi$, i.e., the subgroup that stabilizes the state that is determined by $\psi$.
  Let $X_{\psi,\pi}=G/G_{\psi,\pi}$ and assume that $dX_{\psi,\pi}$
  is  a positive invariant Borel  measure on $X_{\psi,\pi}$. If $\sigma:X_{\psi,\pi}\longrightarrow G$
  is a global Borel section then  $X_{\psi,\pi}$
  parameterizes a family of vectors in $\mathcal{H}$ in the following way
\begin{equation}
X_{\psi,\pi} \ni x \longmapsto \pi(\sigma(x))\psi
\label{e1}
\end{equation}
The vectors that constitute such a family are called generalized coherent states or Perelomov coherent states. For simplicity
 we will refer to these vectors as coherent states (CS). We will adopt the bra-ket notation and denote any vector $v\in \mathcal{H}$ by $|v\rangle$. 
By abuse of notation we will denote the coherent state that corresponds to $x \in X_{\psi,\rho}$ by the following symbols
\begin{equation}
|x\rangle=|\pi(\sigma(x))\psi\rangle=|\pi(x)\psi\rangle
\label{e2}
\end{equation}
 $\psi$ is called admissible if \begin{equation}
C_{\psi,\pi}=\int_{X_{\psi,\pi}}|\langle \psi|x\rangle |^2  dX_{\psi,\pi}(x)<\infty
\label{e3}
\end{equation}
Assuming admissibility, the coherent states family $\left\{|x\rangle \right\}_{x \in X_{\psi,\pi}}$ induces a resolution of the identity
\begin{equation}
\mathbb{I}=\frac{\| \psi \|^2}{C_{\psi,\pi}}\int_{X_{\psi,\pi}}|x\rangle \langle x|  dX_{\psi,\pi}(x)
\label{e4}
\end{equation}
i.e., for any $u,v\in \mathcal{H}$
\begin{equation}
\langle u|v\rangle=\frac{\| \psi \|^2}{C_{\psi,\pi}}\int_{X_{\psi,\pi}}\langle u|x\rangle \langle x|v\rangle  dX_{\psi,\pi}(x)
\label{e5}
\end{equation}
and
\begin{equation}
|v\rangle=\frac{\| \psi \|^2}{C_{\psi,\pi}}\int_{X_{\psi,\pi}} \langle x|v\rangle |x\rangle dX_{\psi,\pi}(x)
\label{e6}
\end{equation}
We shall now describe how one can vary the group and the representation in such a  way that leads to a continuous change of the coherent states. For that we shall use group contraction\cite{Inonu-Wigner53,Segal,Saletan61,Mic,Dooley85,Ricci}. 
Group contraction enables one to obtain a Lie group as a certain limit of another group.  Let $EA$ be the direct product group of the group of real numbers, $\mathbb{R}$,  and the affine group, A; we shall call this group the extended affine. Coherent states for the extended affine group coincide with coherent states for the affine group. Let $H$ be the Heisenberg group. We will show how to  contract the extended affine  group and its unitary irreducible representations to the Heisenberg group and its  unitary irreducible representations. This means that we shall have a continuous family of unitary irreducible representations  $\pi_{\epsilon}: G_{\epsilon}\longrightarrow \mathcal{U}(\mathcal{H})$ with $$G_{\epsilon}\simeq \begin{cases}
EA& \epsilon \neq 0\\
H& \epsilon =0
\end{cases}$$ 
Moreover we shall contract the wavelet analysis to Gabor anlysis in the following sense.
\begin{theorem}
For any infinite dimensional unitary irreducible representation $\pi_0$ of $G_0=H$ on $L^2(\mathbb{R})$ there is a continuous family of unitary irreducible representations 
$\pi_{\epsilon}: G_{\epsilon}\longrightarrow \mathcal{U}(L^2(\mathbb{R}))$ with $\epsilon \in [0,1]$, that realizes a contraction to $\pi_0$ and such that the corresponding formulas (\ref{e1}-\ref{e6}) vary continuously. In addition $\epsilon\neq 0$ leads to a  wavelets analysis on $L^2(\mathbb{R})$ and  $\epsilon=0$ to  Gabor analysis on $L^2(\mathbb{R})$.\label{The1}
\end{theorem}
\noindent For more  precise statements see Proposition \ref{Prop2}, Theorem \ref{Proposition5}, and Proposition \ref{Prop4}.
In many applications one prefers to work in a discretized setup. 
The discrete analogue for the resolution of the identity is given by the notion of tight frame, see, e.g., \cite{ali2013coherent,Dau}.
Let $\mathcal{H}$ be a Hilbert space. A sequence of vectors in $\mathcal{H}$, $\left\{\psi_n \right\}_{n=0}^{\infty}$, is
called a tight frame if there exists a positive number $A$ such that 
\begin{equation}
\frac{1}{A}\sum_{n=0}^{\infty}|\psi_n \rangle \langle \psi_n|=  \mathbb{I}
\label{e7}
\end{equation}
In that case, for any $f\in \mathcal{H}$ we  have the decomposition
\begin{equation}
|f\rangle=\frac{1}{A}\sum_{n=0}^{\infty}(\langle \psi_n|f\rangle)|\psi_n \rangle.
\label{e8}
\end{equation}
In \cite{Dau} tight frames for unitary irreducible representations of the affine group and the  Weyl-Heisenberg group were constructed by choosing a  lattice of coherent states. That is, for  
  $\pi: G\longrightarrow \mathcal{U}(L^2(\mathbb{R}))$  a unitary irreducible representation  of one of these groups, for certain  $\psi$ in   $L^2(\mathbb{R})$, and by picking  certain ``lattice points" $\{x_{n,m}\in X_{\psi,\pi}|(n,m)\in \mathbb{Z}^2\}$, they  constructed a tight frame in $L^2(\mathbb{R})$ by   
  $$\left| \psi^{\pi} _{(n,m)}\right\rangle=\pi(\sigma(x_{n,m}))\psi, \hspace{5mm} (n,m)\in \mathbb{Z}^2$$
For the  Weyl-Heisenberg group they showed how one can find an  admissible $\psi\in L^2(\mathbb{R})$ that is compactly supported. Such admissible functions are better suited for analysis of signals that are spread over a finite time interval. For the
two-dimensional affine group they showed how one can find an  admissible $\psi\in L^2(\mathbb{R})$   such that its Fourier transform is compactly supported. Such admissible functions are better suited for analysis of signals that are band limited, i.e., signals
 with bounded frequencies. In some sense these two analyses are complementary to each other but it is not clear how are they related.

Here we build continuous families of tight frames that interpolate Gabor frames and wavelets frames.
Keeping the same notations as in Theorem \ref{The1} we show that:

\begin{theorem}
For any infinite dimensional unitary irreducible representation $\pi_0$ of $G_0=H$ on $L^2(\mathbb{R})$ there is a continuous family of unitary irreducible representations  
$\pi_{\epsilon}: G_{\epsilon}\longrightarrow \mathcal{U}(L^2(\mathbb{R}))$, as in Theorem \ref{The1},  a family of $\psi_{\epsilon}\in L^2(\mathbb{R})$, and sections  $\sigma_{\epsilon}:X_{\psi_{\epsilon},\pi_{\epsilon}}\longrightarrow G_{\epsilon}$   such that:\\
1. $\pi_{\epsilon}$ realizes a contraction to $\pi_0$.\\
2. $\left\{\left| (\psi_{\epsilon})^{\pi_{\epsilon}} _{(n,m)}\right\rangle|(n,m)\in \mathbb{Z}^2\right\}$ constitute a tight frame in $L^2(\mathbb{R})$. \\
3. $\lim_{\epsilon \longrightarrow 0^+}\left| (\psi_{\epsilon})^{\pi_{\epsilon}} _{(n,m)}\right\rangle=\left| (\psi_{0})^{\pi_{0}} _{(n,m)}\right\rangle$.\\
4. The corresponding formulas (\ref{e7}-\ref{e8}) vary continuously.\label{the2}
\end{theorem}
\noindent Notice that we were able to contract both compactly supported  $\psi_{\epsilon}$ and band limited $\psi_{\epsilon}$. 
For more precise statements see Theorem \ref{propo7} (and its corollaries), Proposition \ref{propIV5}, Proposition\ref{prop9}, Proposition\ref{prop10} (and their corollaries).

Contraction of coherent states and resolutions of the identity have appeared in several cases before (e.g. see \cite{ali2013coherent,Renaud}) but this seems to be the first instance of contraction of frames. Recently, a transform relating dual pairs of wavelets  frames  to dual pairs of Gabor frames and vice versa was considered at \cite{Ole}. Deformation and perturbation of Gabor frames is an active research area\cite{Feichtinger,deGosson2015196,Grochenig}; our method and specifically Theorem \ref{the2} suggests a new approach for deforming Gabor frames.

This paper is organized as follows: In sections \ref{sec3} and \ref{sec4}  we describe the unitary irreducible representations and coherent states
of the Heisenberg group and of the extended affine group. In section \ref{sec5} we show how to contract the extended affine group to the
Heisenberg group. We then show how to contract the unitary irreducible representations, coherent states families, resolutions of the identity and tight frames. In section \ref{sec6}
we show how one can take a family of tight frames that is based upon a compactly supported function and  transform it to a family of tight
frames that is based upon a function with compactly supported Fourier transform. Concluding remarks are presented in section \ref{sec7}.

\section{the heisenberg group: unitary irreducible representations, coherent states and frames}\label{sec3}
\noindent We recall  that the  Heisenberg group is the semidirect product $H\equiv \mathbb{R}^2\rtimes_{\varphi}\mathbb{R}$ where  $\varphi:\mathbb{R} \longrightarrow Aut(\mathbb{R}^2)$
 is defined by $\varphi_{\alpha}\left(\left(\begin{array}{c}
x_1\\
x_2\end{array}\right)\right)= \left(\begin{array}{c}
x_1+\alpha x_2\\
x_2 \end{array}\right)$ for every $\alpha \in \mathbb{R}$ and $\left(\begin{array}{c}
x_1\\
x_2\end{array}\right) \in \mathbb{R}^2$. And the product of $(\alpha,\vec{v}), (\beta,\vec{u}) \in H$ is
given by:  $(\alpha,\vec{v})(\beta,\vec{u})=(\alpha + \beta, \varphi_{\alpha}(\vec{u})+\vec{v})$.
%\subsection{Representation Theory of the Heisenberg group}
Using "the Mackey machine"\cite{Mac}, one can show that,
for every $A \in \mathbb{R}^{*},B\in \mathbb{R}$, we have a  unitary irreducible representation
\begin{eqnarray}\nonumber
&&\eta^{A,B}:H \longrightarrow \mathcal{U}(L^2(\mathbb{R},dx))\\ \nonumber
&& (\eta^{A,B}((c, \left(\begin{array}{c}
v_1\\
v_2\end{array}\right)))f)(x)=e^{i\left[A(v_1+xv_2)+Bv_2  \right]}f(c+x)
\end{eqnarray}
and up to equivalence of representations, any  infinite dimensional unitary irreducible representation of $H$ is of that form. Moreover $\eta^{A_1,B_1}\cong \eta^{A_2,B_2}$ if and only if $A_1=A_2$.

\subsection{Coherent states of the Heisenberg group}
\noindent We recall briefly  the construction of Perelomov coherent states for the unitary irreducible representations $\eta^{A,B}$. For further details see\cite{perelomov2012generalized}. Any nonzero $\psi\in L^{2}(\mathbb{R},dx)$ is admissible and for a generic  $\psi\in L^{2}(\mathbb{R},dx)$  the isotropy subgroup relative to the representation $\eta^{A,B}$ is given by $Z(H)\equiv\left\{(0, \left(\begin{array}{c}
v_1\\
0\end{array}\right))|v_1 \in \mathbb{R} \right\}$, the center of $H$.
Let $X_{H}=X_{\psi,\eta^{A,B}}$ be the  homogeneous space $ H/Z(H)$ and let $\sigma_H :X_H\longrightarrow H$ denote the natural Borel section
$\sigma_H((q, \left(\begin{array}{c}
z\\
p\end{array}\right))Z(H)) =(q, \left(\begin{array}{c}
\frac{qp}{2}\\
p\end{array}\right))$.
The plane $\mathbb{R}^2$ is naturally  identified with   $X_H$ by the following rule 
\begin{eqnarray}\nonumber
(q,p)\longmapsto \sigma_H((q, \left(\begin{array}{c}
z\\
p\end{array}\right))Z(H)) =(q, \left(\begin{array}{c}
\frac{qp}{2}\\
p\end{array}\right))
\end{eqnarray}
and parameterizes  the family of coherent states by 
\begin{eqnarray}\nonumber
&&(q,p)\longmapsto |\psi_{\sigma_H(q,p)}\rangle =\eta^{A,B}\left(q, \left(\begin{array}{c}
\frac{qp}{2}\\
p\end{array}\right)\right)\psi
\end{eqnarray}
The Lebesgue measure on $\mathbb{R}^2$ is an invariant measure for the action of $H$.
In cartesian coordinates of $\mathbb{R}^2$ equation (\ref{e4}) becomes
\begin{eqnarray}\label{2.9}
&&\mathbb{I}=\frac{\|\psi \|_{L^2(\mathbb{R},dx)}^2}{C_{\psi,\eta^{A,B}}} \int_{\mathbb{R}^2} |\psi_{\sigma_H(q,p)}\rangle \langle \psi_{\sigma_H(q,p)} |dqdp
\end{eqnarray}
A straightforward calculation shows that
\begin{eqnarray}\label{2.9}
C_{\psi,\eta^{A,B}}=\frac{2\pi\|\psi \|_{L^2(\mathbb{R},dx)}^4}{|A|}
\end{eqnarray}
For the canonical choice of the fiducial state, i.e., $\psi(x)=\pi^{-\frac{1}{4}}e^{\frac{-x^2}{2}}$,
and the natural representation $\eta^{A=1,B=0}$
we have $|\psi_{\sigma_H(-q,p)}\rangle(x)=\pi^{-\frac{1}{4}}e^{-ip(\frac{q}{2}-x)}e^{-\frac{(x-q)^2}{2}}$.
\subsection{Frames for the Heisenberg group. }
\noindent Let $q_0$ and  $p_0$ be real numbers such that $|q_0p_0|<2\pi$. Consider the discrete subset of $H$
which is given by $H_{q_0,p_0}=\left\{\left( nq_0, \left(\begin{array}{c}
\frac{mnq_0p_0}{2}\\
mp_0\end{array}\right)\right)|n,m \in \mathbb{Z} \right\}$. It is well known (e.g., \cite{perelomov2012generalized})
that for any $ 0\neq\psi \in L^2(\mathbb{R},dx)$ the sequence
$$\left| \psi^{A,B} _{(n,m)}\right\rangle=\left| \psi^{\eta^{A,B}} _{\sigma_H(nq_0,mp_0)}\right\rangle=\eta^{A,B}\left(nq_0, \left(\begin{array}{c}
\frac{nq_0mp_0}{2}\\
mp_0\end{array}\right)\right)\psi$$
where $n,m \in \mathbb{Z}$ is dense in $L^2(\mathbb{R},dx)$. Note that for $q_0p_0=2\pi$  the sequence
constitutes a Von Neumann lattice which was considerably studied and used in quantum mechanics, see e.g., \cite{Von,Barg,Per2,Bac,Bal}.
The following theorem follows from \cite{Dau}; for self-containment we include a variant of the proof given there.  
\begin{theorem}\label{pr1}
Let $0\neq\psi \in L^2(\mathbb{R},dx)$ be such that its support is contained in the interval $[-L,L]$. Let
$p_0=\frac{\pi}{AL}$ and fix $q_0$ such that $|q_0p_0|<2\pi$, i.e., such that $|q_0|<2|A|L$.
Suppose that there exists a positive constant $\chi$ such that for all $x\in \mathbb{R}$
\begin{equation}
\sum_{n\in \mathbb{Z}}\left|\psi(x+nq_0) \right|^2=\chi .
\end{equation}
Then the sequence $\left\{\left| \psi^{A,B} _{(n,m)}\right\rangle\right\}_{n,m\in \mathbb{Z}}$ constitutes
a tight frame with frame constant that is equal to $2L\chi $ i.e.,
\begin{eqnarray}
&&\sum_{n,m\in \mathbb{Z}}^{\infty}\left| \psi^{A,B} _{(n,m)}\right\rangle \left\langle \psi^{A,B} _{(n,m)}\right|=  2L\chi \mathbb{I}
\end{eqnarray}
$($For construction of such $\psi$ see \cite{Dau}$)$.
\end{theorem}

\begin{proof}
Let $f \in L^2(\mathbb{R},dx)$. Recall that  Parseval's theorem implies that for $f \in L^2([-L,L],dx)\subset L^2(\mathbb{R},dx)$
we have $2L\|f\|_{L^2(\mathbb{R},dx)}^2=\sum_{n\in \mathbb{Z}}|C_n|^2$ where $C_n=\int_{\mathbb{R}}f(x)e^{-i\frac{n}{L}\pi x}dx$.
Using this we observe that
\begin{eqnarray}\nonumber
&&\sum_{n,m\in \mathbb{Z}}\left|\langle f| \psi^{A,B} _{(n,m)}\rangle\right|^2=\sum_{n,m\in \mathbb{Z}}\left|\int_{\mathbb{R}} \overline{f(x)}e^{iAx\frac{\pi}{AL}m}\psi(x+nq_0)dx \right|^2=\\ \nonumber
&&\sum_{n,m\in \mathbb{Z}}\left|\int_{\mathbb{R}} \overline{f(y-nq_0)}e^{i(y-nq_0)\frac{\pi}{L}m}\psi(y)dy \right|^2=\sum_{n,m\in \mathbb{Z}}\left|\int_{\mathbb{R}} \overline{f(y-nq_0)}e^{iy\frac{\pi}{L}m}\psi(y)dy \right|^2=\\ \nonumber
&&{2L}\sum_{n\in \mathbb{Z}}\int_{\mathbb{R}} \left|f(y-nq_0)\psi(y)\right|^2dy =  {2L}\sum_{n\in \mathbb{Z}}\int_{\mathbb{R}} \left|f(x)\right|^2\left|\psi(x+nq_0)\right|^2dx =\\ \nonumber
&&  {2L}\int_{\mathbb{R}} \left|f(x)\right|^2\sum_{n\in \mathbb{Z}}\left|\psi(x+nq_0)\right|^2dx =  {2L}\chi \int_{\mathbb{R}} \left|f(x)\right|^2dx = {2L}\chi  \|f\|^2_{L^2(\mathbb{R},dx)}
\end{eqnarray}
(we  used the change of variables $y=x+nq_0$ and we interchanged the order of the  sum and the integral which is justified by the compactness of the support of $\psi$.)
\end{proof}

\section{the extended affine group: unitary irreducible representations and coherent states}\label{sec4}
\subsection{The affine group}
\noindent Let $A$ be the affine group realized  as the matrix group $$\left\{\left(\begin{array}{cc}
\alpha & \beta  \\
0 & 1 \end{array}\right)|\alpha \in \mathbb{R}^+,\beta\in\mathbb{R}\right\}  $$ Using "the Mackey machine" \cite{Mac},
one can show that for any $b\in\mathbb{R}^*$ the formula
\begin{eqnarray}\nonumber
\left(U^b\left(\begin{array}{cc}
\alpha & \beta  \\
0 & 1 \end{array}\right)f\right)(x)=\frac{1}{\sqrt{\alpha}}f(\frac{x+b\beta}{\alpha})
\end{eqnarray}
defines a unitary irreducible representation of the affine group on the subspace of  $L^2(\mathbb{R},dx)$ which consists of functions
that their Fourier transforms are supported in the positive half of the real line.
It is well known (e.g.,  \cite{ali2013coherent}) that $\psi\in L^2(\mathbb{R},dx)$ is admissible if and only if
\begin{eqnarray}\nonumber
\int_{0}^{\infty}|\mathcal{F}(\psi)(x)|^2\frac{dx}{x}< \infty
\end{eqnarray}
where $\mathcal{F}(\psi)$ is the Fourier transform of $\psi$.
Let $\psi$ be  admissible, then
\begin{eqnarray}\nonumber
&& C_{\psi,U^b}= \frac{2\pi}{|b|}\|\psi\|_{L^2(\mathbb{R},dx)}^2\int_{0}^{\infty}|\mathcal{F}(\psi)(x)|^2\frac{dx}{x}< \infty
\end{eqnarray}
and $G_{\psi,U^b}$ is the trivial group. The invariant measure on $X_{\psi,U^b}=A$ in terms of the natural
coordinates is given by $\frac{d\alpha d\beta}{\alpha^2}$.  The coherent states family is parameterized by the affine group, $A$, according to:
\begin{eqnarray}\nonumber
&&|\psi_{\alpha,\beta}^b\rangle(x) =\left(U^b\left(\begin{array}{cc}
\alpha & \beta  \\
0 & 1 \end{array}\right)\psi \right)(x)=\frac{1}{\sqrt{\alpha}}\psi(\frac{x+b\beta}{\alpha})
\end{eqnarray}
and satisfies
\begin{eqnarray}\label{3.4}
&&\mathbb{I}=\frac{\|\psi \|_{L^2(\mathbb{R},dx)}^2}{C_{\psi,U^b}}\int_{\mathbb{R}^+\times \mathbb{R}}|\psi^b_{\alpha,\beta}\rangle \langle \psi^b_{\alpha,\beta} |\frac{d\alpha d\beta }{\alpha^2}
\end{eqnarray}
There are two inequivalent infinite dimensional unitary irreducible representations of the affine group, one for each sign of $b$. A frequent choice in the literature is $U^{b=-1}$. In this representation the CS are given by
\begin{eqnarray}\label{3.5}
|\psi_{\alpha,\beta}\rangle(x) =|\psi_{\alpha,\beta}^{b=-1}\rangle(x)=\frac{1}{\sqrt{\alpha}}\psi(\frac{x-\beta}{\alpha})
\end{eqnarray}
Any admissible function $\psi$ is called a wavelet and for any  $f\in L^2(\mathbb{R},dx)$ its wavelet transform with respect to $\psi$ is given by :
\begin{eqnarray}\nonumber
\langle \psi_{\alpha,\beta} |f\rangle=\frac{1}{\sqrt{\alpha}}\int_{-\infty}^{\infty}\overline{\psi(\frac{(x-\beta)}{\alpha})}f(x)dx
\end{eqnarray}

\subsection{Representation theory of the extended affine group }\label{secb}
\noindent We define the extended affine group to be the direct product of the  affine group,
$A$, with the additive group of the real numbers, $\mathbb{R}$. We denote this group
by $EA\equiv A \oplus \mathbb{R}$. The product in $EA$ is given by $(\alpha, \beta, \gamma)(x,y,z)=(\alpha x, \alpha y+\beta, \gamma+z)$.
  For $a\in \mathbb{R}$ let $\chi_{a}(x)=e^{iax}$ be a unitary character of $ \mathbb{R}$.
  Let $\mathcal{H}$ denote the subspace of $L^2(\mathbb{R},dx)$ which consists of all functions
  whose Fourier transforms vanish on the negative half of $\mathbb{R}$.   For every
  $b\in \mathbb{R}^{*}$, $a\in \mathbb{R}$ we have the   unitary irreducible representation $U^b\otimes \chi_{a}:EA \longrightarrow \mathcal{U}(\mathcal{H})$ which is given by:
      \begin{eqnarray}\nonumber
      (U^b\otimes \chi_{a}((\alpha, \beta, \gamma))f)(x)=e^{ia \gamma}\frac{1}{\sqrt{\alpha}}f(\frac{x+b\beta}{\alpha})
      \end{eqnarray}
  Up to equivalence, any  infinite dimensional unitary irreducible representation of $EA$ is of that form. In addition  $\rho^{a,b}\simeq \rho^{a',b'}$ if and only if $a=a'$. 
For contraction of the unitary irreducible representations of $EA$ we need a different  realization which is better suited for the contraction procedure.
We obtain  this realization as follows.
We define the linear invertible maps $S_1,S_2,S_3(\epsilon)$ by:\\
1.  $S_1:\mathcal{H} \xrightarrow{\sim} L^2(\mathbb{R}^+,dx)$, $S_1(f)=\mathcal{F}(f)|_{\mathbb{R^+}}$  where
  $\mathcal{F}(f)(w)=\frac{1}{\sqrt{2\pi}}\int_{\mathbb{R}}f(x)e^{-iwx}dx$.\\ 
2.  $S_2:L^2(\mathbb{R}^+,dx) \xrightarrow{\sim}L^2(\mathbb{R}^+,\frac{1}{x}dx)$ where $S_2(f)(x)=\sqrt{x}f(x)$.\\
3.   For $\epsilon >0$,  $S_3(\epsilon):L^2(\mathbb{R}^+,\frac{1}{x}dx) \xrightarrow{\sim}L^2(\mathbb{R},dx)$, $S_3(\epsilon)(f)=f\circ \psi_{\epsilon}$, where $\psi_{\epsilon}(x)=e^{-\epsilon x}$.\\
We also define $T_{\epsilon}=S_3(\epsilon)\circ S_2 \circ S_1$ 
 and use it as intertwining operator  in the following way.
For each   $\epsilon \in (0,1]$, $b\in \mathbb{R}^{*}$ and $a\in \mathbb{R}$, we intertwine  $U^b\otimes \chi_{a}$ with $T_{\epsilon}$ to get the equivalent representation
\begin{eqnarray}\nonumber
&& \rho_{\epsilon}^{a,b}:EA \longrightarrow \mathcal{U}(V_{\epsilon})\\ \nonumber
&& (\alpha, \beta, \gamma) \longmapsto T_{\epsilon} \circ (U^b\otimes \chi_{a})(\alpha, \beta, \gamma) \circ T^{-1}_{\epsilon}
\end{eqnarray}
Explicitly  for every $f \in V_{\epsilon}$,
\begin{equation}\nonumber
\rho_{\epsilon}^{a,b}(f)|_x=e^{ia\gamma}e^{ib\beta e^{-\epsilon x}}f(-\frac{\log{\alpha}}{\epsilon}+x)
\end{equation}

\subsection{Coherent states of the extended affine group}
\noindent In this section we describe  Perelomov coherent states for the unitary irreducible representations of EA.
One can show that for a generic normalized function $\psi\in L^{2}(\mathbb{R}^+,\frac{dx}{x})$ the isotropy subgroup relative to
the representation $\rho^{a,b}$ is given by $Z(EA)=\left\{\left(1,0,\gamma\right)|\gamma \in \mathbb{R} \right\}$, the center
of $EA$. Let $\psi\in L^{2}(\mathbb{R}^+,\frac{dx}{x})$  be a nonzero  admissible fiducial state. Let $X_{EA}$ denote the homogeneous
space $EA/Z(EA)$ and let $d_{X_{EA}}$ be a positive invariant Borel measure on $X_{EA}$.
Let $\sigma_{EA} :X_{EA}\longrightarrow EA$ be a global Borel section.
The set $\mathbb{R}^+\times\mathbb{R}$ naturally  parameterizes $X_{EA}$ by the following rule
\begin{eqnarray}\nonumber
(\alpha, \beta)\longmapsto \sigma_{EA}(\alpha, \beta,0)Z(EA).
\end{eqnarray}
In these coordinates (\ref{e4}) becomes
\begin{eqnarray}\label{3.12}\nonumber
&&\mathbb{I}= \frac{\|\psi\|^2}{C_{\psi,\rho^{a,b}}}\int_{\mathbb{R}^+\times \mathbb{R}} |\psi^{a,b}_{\sigma_{EA} \left(
\alpha, \beta, 0 \right)}\rangle \langle\psi^{a,b}_{\sigma_{EA} \left(
\alpha, \beta, 0 \right)}|\frac{d\alpha d\beta}{\alpha^2}
\end{eqnarray}
where
\begin{eqnarray}\nonumber
|\psi^{a,b}_{\sigma_{EA} \left(
\alpha, \beta, \gamma \right)}\rangle\equiv\rho^{a,b}(\sigma_{EA} \left((
\alpha, \beta, \gamma \right)Z(EA)) )\psi
\end{eqnarray}
For admissible  $\psi\in L^{2}(\mathbb{R}^+,\frac{dx}{x})$, a straightforward calculation shows that
\begin{eqnarray}\nonumber
&&C_{\psi,\rho^{a,b}}= \frac{2\pi}{|b|}\|\psi\|_{L^{2}(\mathbb{R}^+,\frac{dx}{x})}^2 \|\psi\|_{L^{2}(\mathbb{R}^+,\frac{dx}{x^2})}^2
\end{eqnarray}
We construct frames for the unitary irreducible representations of $EA$ in section \ref{sec5}.

\section{Contracting the extended affine group and its coherent states}\label{sec5}
\noindent In this section we define contraction of Lie groups and
their representations and show how the extended affine group along with its unitary irreducible representations contract to the Heisenberg
group and its unitary irreducible representations.
Moreover we contract the corresponding coherent state families,  resolutions of the identity, and tight frames.

\subsection{The contraction of $EA$ to $H$}
\noindent Contraction of Lie algebras was first introduced by  Segal \cite{Segal} and  {I}n\"{o}n\"{u} and Wigner \cite{Inonu-Wigner53}, as a natural notion for limit of Lie algebras.
If one wants to define contraction for Lie groups, there are several approaches, see for example \cite{Mic,Dooley85,Ricci}.
In our case (the contraction of $EA$ to $H$), since both groups are exponential, lifting the Lie algebras contraction using the exponential map is
 straightforward and we obtain the following definition:
\begin{definition}\label{deIV1}
Let $G_1,G_0$ be two Lie groups. Suppose that for every
$\epsilon \in (0,1]$ we have a diffeomorphism $P_{\epsilon} : G_0 \longrightarrow G_1$  such that the following
conditions hold:
\begin{enumerate}
	\item $P_{\epsilon}(e_0) = e_1$ where $e_0, e_1$ are the identity elements of $G_0, G_1$, respectively,
		\item For every $x, y \in G_0$
\begin{equation}\nonumber
x \cdot y = \lim_{\epsilon \longrightarrow 0^+} P_{\epsilon}^{-1}(P_{\epsilon}(x)\cdot P_{\epsilon}(y)),
\end{equation}
\end{enumerate}
then we say that $G_0$ is the contraction of $G$ by $P_{\epsilon} $ and we denote it by $G_1\stackrel{P(\epsilon)}{\longrightarrow} G_{0}$.
\end{definition}

\begin{proposition}
For $\epsilon \in (0,1]$ let $P_{\epsilon}:H\longrightarrow EA$ be defined by
\begin{eqnarray}\nonumber
&& P_{\epsilon}\left(x, \left(\begin{array}{c}
y\\
z\end{array}\right)\right)=\left(\begin{array}{ccc}
e^{-\epsilon x} &, z(\frac{e^{-\epsilon x}-1}{ -\epsilon x}) &, z+\epsilon (y-\frac{ xz}{2})\end{array}\right)
\end{eqnarray}
for $x\neq 0$ and
\begin{eqnarray}\nonumber
& P_{\epsilon}\left(0, \left(\begin{array}{c}
y\\
z\end{array}\right)\right)&=\lim_{x \longrightarrow 0}\left(\begin{array}{ccc}
e^{-\epsilon x} &, z(\frac{e^{-\epsilon x}-1}{ -\epsilon x}) &, z+\epsilon (y-\frac{ xz}{2})\end{array}\right)=\left(\begin{array}{ccc}
1 &, z &, z+\epsilon y\end{array}\right)
\end{eqnarray}

Then $EA\stackrel{P(\epsilon)}{\longrightarrow} H$  i.e., $P_{\epsilon}$ realizes the contraction of $EA$ to $H$.
\end{proposition}
\begin{proof}
We first note that $P_{\epsilon}$ is smooth and note that
\begin{eqnarray}\nonumber
&& P^{-1}_{\epsilon}\left(\alpha, \beta, \gamma \right)= \left(-\frac{1}{\epsilon}\ln \alpha, \left(\begin{array}{c}
\frac{1}{\epsilon}\left(\gamma-\beta\frac{\ln \alpha}{\alpha-1}-\beta\frac{\ln^2 \alpha}{2(\alpha-1)} \right)\\
\beta\frac{\ln \alpha}{\alpha-1}\end{array}\right)\right)
\end{eqnarray}
for $x\neq 0$ and
\begin{eqnarray}\nonumber
&& P^{-1}_{\epsilon}\left(\alpha, \beta, \gamma \right)= \left(0, \left(\begin{array}{c}
\frac{1}{\epsilon}\left(\gamma-\beta \right)\\
\beta\end{array}\right)\right)
\end{eqnarray}
Obviously $P_{\epsilon}$ is a diffeomorphism that preserves the identity element.
A straightforward calculation shows that for any $x,y \in H$
\begin{equation}\nonumber
x \cdot y = \lim_{\epsilon \longrightarrow 0^+} P_{\epsilon}^{-1}(P_{\epsilon}(x)\cdot P_{\epsilon}(y))
\end{equation}
where the product on the left hand side is the one in $H$ and on the right hand side the one in $EA$. Hence the Heisenberg group, $H$, is the contraction of the extended affine, $EA$, by $P_{\epsilon}$ according to definition  \ref{deIV1}.
\end{proof}

\subsection{The contraction of $\rho^{a,b}$ to $\eta^{A,B}$}
\noindent In \cite{Subag12}, for any representation of the Lie algebra of the Heisenberg group that arises from a unitary irreducible
representation of the Heisenberg group, a family of representations of the Lie algebra of the extended affine that contracts to the given representation was constructed.
In the following we explain the analogous picture at the level of the representations of the groups themselves.

 From now on we  fix a representation $\eta^{A,B}$ of the Heisenberg group on $L^2(\mathbb{R})$ and let $a(\epsilon)=a_0+\frac{A}{\epsilon}, b(\epsilon)=b_0-\frac{A}{\epsilon}$ such that $a_0+b_0=B$.

\begin{proposition}\label{Prop2}
\textbf{$[$Contraction of representations$]$} 
 Consider the family of representations of $EA$ on $L^2(\mathbb{R})$
that is given by  $\left\{        \rho_{\epsilon}^{a(\epsilon),b(\epsilon)}         \right\}_{\epsilon \in (0,1]}$.
Then   $\eta^{A,B}$ is  the strong contraction of the family of representations
 $\left\{         \rho_{\epsilon}^{a(\epsilon),b(\epsilon)}        \right\}_{\epsilon \in (0,1]}$ in the following sense:
\begin{enumerate}
  \item \textbf{(Pointwise convergence)} For any $f\in L^2(\mathbb{R})$, any $x \in \mathbb{R}$, and any $g\in H$
  $$ \lim_{\epsilon \longrightarrow 0^+}\left(\rho_{\epsilon}^{a(\epsilon),b(\epsilon)}(P_{\epsilon}(g))f\right)(x)=\left(\eta^{A,B}(g)f\right)(x).$$
  \item \textbf{(Norm convergence)} For any $f\in L^2(\mathbb{R})$  and any $g\in H$
  $$ \lim_{\epsilon \longrightarrow 0^+}\left(\rho_{\epsilon}^{a(\epsilon),b(\epsilon)}(P_{\epsilon}(g))f\right)=\left(\eta^{A,B}(g)f\right).$$
\end{enumerate}
\end{proposition}

\begin{proof}
We observe that for  any $f\in L^2(\mathbb{R})$, any $x \in \mathbb{R}$, and any $\left( c, \left(\begin{array}{c}
v_1\\
v_2\end{array}\right)\right)\in H$
\begin{eqnarray}\label{4.199}\nonumber
&& \lim_{\epsilon \longrightarrow 0^+}\left(\rho_{\epsilon}^{a(\epsilon),b(\epsilon)}\left(P_{\epsilon}\left( c, \left(\begin{array}{c}
v_1\\
v_2\end{array}\right)\right) \right)f\right)(x)= \\ \nonumber
&& \lim_{\epsilon \longrightarrow 0^+} e^{ia(\epsilon)[ v_2+\epsilon(v_1-\frac{cv_2}{2}) ]+ib(\epsilon)v_2(\frac{e^{-\epsilon c}-1}{-\epsilon c}) e^{-\epsilon x}}f(-\frac{\log{e^{-\epsilon c}}}{\epsilon}+x) \\ \nonumber
&& = e^{i[Bv_2+A(v_1+v_2x)]}f(c+x)= \eta^{A,B}\left( c, \left(\begin{array}{c}
v_1\\
v_2\end{array}\right)\right)  (f)(x)
\end{eqnarray}
For condition (2) we observe that for any $x \in \mathbb{R}$
\begin{eqnarray}\nonumber
&&\left | \rho_{\epsilon}^{a(\epsilon),b(\epsilon)}\left(P_{\epsilon}\left( c, \left(\begin{array}{c}
v_1\\
v_2\end{array}\right)\right) \right)f(x)-\eta^{A,B}\left( c, \left(\begin{array}{c}
v_1\\
v_2\end{array}\right)\right) f(x)\right |^2 
 \leq2|f(x+c)| ^2
\end{eqnarray}
Since  $f\in L^2(\mathbb{R},dx)$  then $2|f(x+c)| ^2\in L^1(\mathbb{R},dx)$  and by  Lebesgue's dominated convergence theorem  
{\small\begin{eqnarray}\nonumber
&&\hspace{2mm}\lim_{\epsilon \longrightarrow 0^+}\left\| \rho_{\epsilon}^{a(\epsilon),b(\epsilon)}\left(P_{\epsilon}\left( c, \left(\begin{array}{c}
v_1\\
v_2\end{array}\right)\right) \right)f-\eta^{A,B}\left( c, \left(\begin{array}{c}
v_1\\
v_2\end{array}\right)\right) f\right\|^2 =0
\end{eqnarray}}
\end{proof}

\subsection{Contraction of coherent states}
\noindent We can reinterpret the contraction  $EA\stackrel{P(\epsilon)}{\longrightarrow} H$ as a continuous family of  Lie groups
 $\left\{G_{\epsilon}\right\}_{\epsilon \in [0,1]}$,
such that for  $\epsilon \in (0,1]$,  $G_{\epsilon}$  is isomorphic to $EA$ and and  $G_0=H$. More specifically, for
$\epsilon \in (0,1]$ let $G_{\epsilon}$ be the group with $\mathbb{R}\times \mathbb{R}^2 $  as the underlying smooth manifold and a product that for any $x,y \in \mathbb{R}\times \mathbb{R}^2 $ is given by
\begin{equation}\nonumber
x \cdot_{\epsilon} y =  P_{\epsilon}^{-1}(P_{\epsilon}(x)\cdot P_{\epsilon}(y))
\end{equation}
where the product on the right hand side is taking place in $EA$. As we already saw,
\begin{equation}\nonumber
\lim_{\epsilon \longrightarrow 0^+}x \cdot_{\epsilon} y = x \cdot_{0} y
\end{equation} where the product on the right hand side is in $H$. Hence, we have a family of  Lie products over
the smooth manifold $\mathbb{R}\times \mathbb{R}^2 $ that vary continuously in $\epsilon$. In the following we use this continuous family of groups
in order to contract coherent states families and tight frames.

The map $P_{\epsilon}:G_{\epsilon} \longrightarrow EA$ is an isomorphism of Lie groups and hence
the map $\widetilde{\rho}_{\epsilon}^{a,b}=\rho_{\epsilon}^{a,b}\circ P_{\epsilon}$ defines a representation of $G_{\epsilon}$ on $L^2(\mathbb{R},dx)$.
For generic $\psi \in  L^2(\mathbb{R},dx)$, the isotropy subgroup of $G_{\epsilon}$ for the representation $\widetilde{\rho}_{\epsilon}^{a,b}$ is $Z(G_{\epsilon})=\left\{ \left( 0, \left(\begin{array}{c}
z\\
0\end{array}\right)\right)| z\in \mathbb{R}\right\}$, the center of  $G_{\epsilon}$.
Let $X_{\epsilon}$ denote the homogeneous space $G/G_{\epsilon}$ and let $\sigma_{\epsilon}:X_{\epsilon}\longrightarrow G_{\epsilon} $ be the global Borel section that is given by
\begin{equation}\nonumber
\sigma_{\epsilon}\left(\left( q, \left(\begin{array}{c}
z\\
p\end{array}\right)\right) Z(G_{\epsilon}) \right)=\left( q, \left(\begin{array}{c}
\frac{qp}{2}\\
p\end{array}\right)\right)
\end{equation}
As before, $\mathbb{R}^2$ parameterizes $X_{\epsilon}$ and hence also the CS according to
\begin{eqnarray}\nonumber
&&(q,p) \longmapsto |\psi_{\epsilon}^{a,b}(q,p)\rangle=\left(\widetilde{\rho}_{\epsilon}^{a,b} \left( q, \left(\begin{array}{c}
\frac{qp}{2}\\
p\end{array}\right)\right)\right)\psi
\end{eqnarray}
The invariant measure on $X_{\epsilon}$ with respect to the action of $G_{\epsilon}$ in the coordinates $(q,p)$ is easily calculated to be
\begin{eqnarray}\nonumber
\frac{e^{\epsilon q}-1}{\epsilon q}dqdp
\end{eqnarray}
And the corresponding resolution of the identity is given by
\begin{eqnarray}\label{5.14}\nonumber
\mathbb{I}=\frac{\|\psi\|_{L^2(\mathbb{R},dx)} ^2}{C_{\psi,\widetilde{\rho}_{\epsilon}^{a,b} }}\int_{\mathbb{R}^2} |\psi_{\epsilon}^{a,b}(q,p)\rangle \langle\psi_{\epsilon}^{a,b}(q,p)|    \frac{e^{\epsilon q}-1}{\epsilon q}dqdp
\end{eqnarray}
\begin{proposition}\label{n4}
For admissible nonzero $\psi\in V_{\epsilon}$
\begin{eqnarray}\nonumber
C_{\psi,\widetilde{\rho}_{\epsilon}^{a(\epsilon),b(\epsilon)} }=&&\frac{2\pi}{|A-b_0\epsilon|} \|\psi\|^2_{L^2(\mathbb{R},dx)} \int_{\mathbb{R}} |\psi(x) |^2   e^{\epsilon x} dx=\frac{2\pi}{|A-b_0\epsilon|} \|\psi\|^2_{L^2(\mathbb{R},dx)}\|\psi\|^2_{L^2(\mathbb{R},e^{\epsilon x}dx)}
\end{eqnarray}
\end{proposition}
\noindent The proof can be obtained by a change of variables.
\begin{theorem}
\textbf{$[$Contraction of coherent states$]$} 
For any $x\in \mathbb{R}$, $(q,p)\in \mathbb{R}^2$
\begin{eqnarray}\nonumber
&&\lim_{\epsilon \longrightarrow 0^+}| \psi_{\epsilon}^{{a(\epsilon),b(\epsilon)}}(q,p) \rangle(x) =| \psi_{\sigma_{H}(q,p)}^{\eta^{A,B}} \rangle(x)
\end{eqnarray}
and
\begin{eqnarray}\nonumber
&&\lim_{\epsilon \longrightarrow 0^+}| \psi_{\epsilon}^{{a(\epsilon),b(\epsilon)}}(q,p) \rangle =| \psi_{\sigma_{H}(q,p)}^{\eta^{A,B}} \rangle
\end{eqnarray}
i.e., the convergence is pointwise and also in  $L^2(\mathbb{R},dx)$.
\label{Proposition5}
\end{theorem}

\begin{proof}
For any $x\in \mathbb{R}$, $(q,p)\in \mathbb{R}^2$ we have
\begin{eqnarray}\nonumber
&&\lim_{\epsilon \longrightarrow 0^+}| \psi_{\epsilon}^{{a(\epsilon),b(\epsilon)}}(q,p) \rangle (x)=\lim_{\epsilon \longrightarrow 0^+}\left(\widetilde{\rho}_{\epsilon}^{a(\epsilon),b(\epsilon)} \circ  \sigma_{\epsilon}\left( q, \left(\begin{array}{c}
0\\
p\end{array}\right)\right)\right)\psi(x) =\\ \nonumber
&&\lim_{\epsilon \longrightarrow 0^+}\left(\rho_{\epsilon}^{a(\epsilon),b(\epsilon)} \circ  P_{\epsilon}\left( q, \left(\begin{array}{c}
\frac{qp}{2}\\
p\end{array}\right)\right)\right)\psi(x) \underbrace{=}_{\text{Prop.}(\ref{Prop2})}\eta^{A,B}\left( q, \left(\begin{array}{c}
\frac{qp}{2}\\
p\end{array}\right)\right)\psi(x) =| \psi_{\sigma_{H}(q,p)}^{\eta^{A,B}} \rangle (x)
\end{eqnarray}
The convergence in norm follows by the same argument that was used in the proof of Proposition \ref{Prop2}.
\end{proof}

\subsection{Contraction of the resolution of the identity}
\noindent For any $\epsilon\in (0,1]$ we have the resolution of the identity that is associated with a representations of $G_{\epsilon} (\simeq EA)$:
{\footnotesize \begin{eqnarray}\nonumber
&&\mathbb{I}=\frac{\|\psi\|_{L^2(\mathbb{R},dx)}^2}{C_{\psi,\widetilde{\rho}_{\epsilon}^{a(\epsilon),b(\epsilon)} }}\int_{\mathbb{R}^2} |\psi_{\epsilon}^{a(\epsilon),b(\epsilon)}(q,p)\rangle \langle\psi_{\epsilon}^{a(\epsilon),b(\epsilon)}(q,p)|    \frac{e^{\epsilon q}-1}{\epsilon q}dqdp\\ \nonumber
&&=\frac{|b(\epsilon)|}{2\pi \|\psi\|_{L^2(\mathbb{R},e^{\epsilon x}dx)}^2}\int_{\mathbb{R}^2} |\psi_{\epsilon}^{a(\epsilon),b(\epsilon)}(q,p)\rangle \langle\psi_{\epsilon}^{a(\epsilon),b(\epsilon)}(q,p)|    \frac{e^{\epsilon q}-1}{\epsilon q}dqdp
\end{eqnarray}}
We would like to interpret its limit when $\epsilon$ goes to zero as
\begin{eqnarray}\nonumber
&\mathbb{I}=&\frac{\|\psi\|_{L^2(\mathbb{R},dx)}^2}{C_{\psi,\eta^{A,B} }} \int_{\mathbb{R}^2} |\psi_{\sigma_H(q,p)}\rangle \langle \psi_{\sigma_H(q,p)} |dqdp=
\frac{|A|}{2\pi \|\psi\|_{L^2(\mathbb{R},dx)}^2} \int_{\mathbb{R}^2} |\psi_{\sigma_H(q,p)}\rangle \langle \psi_{\sigma_H(q,p)} |dqdp
\end{eqnarray}
which is a resolution of the identity that is associated with a representation of $H$.
We formulate it in the following proposition:
\begin{proposition}\label{Prop4}
\textbf{$[$Contraction of resolutions of the identity$]$}
For any $f,g\in L^2(\mathbb{R},dx)$
\begin{eqnarray}\nonumber
\lim_{\epsilon \longrightarrow 0^+} &&\int_{\mathbb{R}^2} \langle f|\psi_{\epsilon}^{a(\epsilon),b(\epsilon)}(q,p)\rangle \langle\psi_{\epsilon}^{a(\epsilon),b(\epsilon)}(q,p)|g\rangle    \frac{e^{\epsilon q}-1}{\epsilon q}dqdp\\ \label{n1}
&&= \int_{\mathbb{R}^2} \langle f|\psi^{\eta^{A,B}}_{\sigma_H(q,p)}\rangle \langle \psi^{\eta^{A,B}}_{\sigma_H(q,p)}|g\rangle dqdp
\end{eqnarray}
\end{proposition}
\begin{proof}
Since \begin{eqnarray}\nonumber
\lim_{\epsilon \longrightarrow0^+}\frac{e^{\epsilon q}-1}{\epsilon q} =1,
\end{eqnarray}
then Proposition \ref{Proposition5} implies
 Proposition \ref{Prop4}  provided that we can interchange the limit and the integral in (\ref{n1}). We  prove Proposition \ref{Prop4} using a  more direct approach. We note that Proposition \ref{Prop4} holds if and only if
\begin{eqnarray}\nonumber
\lim_{\epsilon \longrightarrow 0^+} C_{\psi,\widetilde{\rho}_{\epsilon}^{a(\epsilon),b(\epsilon)} }=C_{\psi,\eta^{A,B} }
\end{eqnarray}
By Proposition \ref{n4} we have
\begin{eqnarray}\nonumber
&&C_{\psi,\widetilde{\rho}_{\epsilon}^{a(\epsilon),b(\epsilon)} }=2\pi  \|\psi\|_{L^2(\mathbb{R},dx)}  \int_{\mathbb{R}}    \frac{|\psi(x)|^2}{|A-b_0\epsilon|} e^{\epsilon x}dx
\end{eqnarray}
and by the monotone convergence theorem we obtain
\begin{eqnarray}\nonumber
\lim_{\epsilon \longrightarrow 0^+}C_{\psi,\widetilde{\rho}_{\epsilon}^{a(\epsilon),b(\epsilon)} }=&&2\pi  \|\psi\|_{L^2(\mathbb{R},dx)}   \int_{\mathbb{R}}\lim_{\epsilon \longrightarrow 0^+}  \frac{|\psi(x)|^2}{|A-b_0\epsilon|} e^{\epsilon x}dx=2\pi \frac{\|\psi\|_{L^2(\mathbb{R},dx)}^4}{|A|}=C_{\psi,\eta^{A,B} }
\end{eqnarray}
\end{proof}
\subsection{Frames with compactly supported fiducial states and their  contractions}
\noindent In this section, we  construct tight frames for the unitary irreducible representations of the extended affine group. We show how the tight frames of the extended affine
group contract (in a sense that is explained below) to tight frames of the Heisenberg group.\newline
\begin{theorem}
\textbf{$[$Tight frames for $EA$$]$}
Under the assumptions of Theorem \ref{pr1},  for any $\epsilon \in (0,1] $ let $EA_{q_0,p_0}(\epsilon)$ be the
discrete subset of  $EA$ that is defined by $\left\{ \left(\alpha_n(\epsilon), \beta_{mn}(\epsilon), \gamma_{mn}(\epsilon) \right) |m,n \in \mathbb{Z}\right\}$ where
\begin{eqnarray}\nonumber
&& \alpha_n(\epsilon)=e^{-\epsilon nq_0}\\ \nonumber
&& \beta_{mn}(\epsilon)=-\frac{2\pi \alpha_n(\epsilon)  m}{b(\epsilon)(e^{\epsilon L}-e^{-\epsilon L})}=-\frac{e^{-\epsilon nq_0}2\pi    m}{b(\epsilon)(e^{\epsilon L}-e^{-\epsilon L})}\\ \nonumber
&& \gamma_{mn}(\epsilon)=\beta_{mn}(\epsilon)\frac{\ln \alpha_n(\epsilon)}{\alpha_n(\epsilon)-1}=\frac{e^{-\epsilon nq_0}2\pi   m}{b(\epsilon)(e^{\epsilon L}-e^{-\epsilon L})}\frac{\epsilon nq_0}{e^{-\epsilon nq_0}-1}.
\end{eqnarray}
The sequence $\left\{ \left|\psi^{\epsilon,a(\epsilon),b(\epsilon)}_{(n,m)} \right\rangle \right\}_{m,n \in \mathbb{Z}}$, where
\begin{eqnarray}\nonumber
&&\left|\psi^{\epsilon,a(\epsilon),b(\epsilon)}_{(n,m)} \right\rangle=\rho^{a(\epsilon),b(\epsilon)}_{\epsilon}(\alpha_n(\epsilon), \beta_{mn}(\epsilon), \gamma_{mn}(\epsilon))\left( Q_{\epsilon}\psi \right)
\end{eqnarray}
with $Q_{\epsilon}(x)=e^{-\frac{1}{2}\epsilon x}$, constitutes a tight frame with frame constant that is equal to $2\frac{\sinh (\epsilon L)}{\epsilon}\chi$.
\label{propo7}
\end{theorem}
\begin{remark}
We refer to $Q_{\epsilon}(x)\psi(x)$ as the new mother wavelet (fiducial state). We emphasize that if one uses $\psi(x)$ as a mother wavelet instead of
$Q_{\epsilon}(x)\psi(x)$, then $\psi(x)$ does not generate a tight frame for $\epsilon \neq 0$.
\end{remark}
\begin{proof}
Let $f \in L^2(\mathbb{R},dx)$. Using the change of variables $y=e^{-\epsilon x}$ and then $z={y}{\alpha_n(\epsilon)}$ we obtain
\begin{eqnarray}\nonumber
\nonumber
&&\sum_{n,m\in \mathbb{Z}}\left|\langle f| \psi^{\epsilon,a(\epsilon),b(\epsilon)}_{(n,m)}\rangle\right|^2=\sum_{n,m\in \mathbb{Z}}\left|\int_{\mathbb{R}} \overline{f(x)}e^{ib(\epsilon)\beta_{mn}(\epsilon)e^{-\epsilon x}}e^{-\frac{\epsilon}{2}(x-\frac{\ln \alpha_n(\epsilon)}{\epsilon})}\psi(x-\frac{\ln \alpha_n(\epsilon)}{\epsilon})dx \right|^2\\ \nonumber
&&=\sum_{n,m\in \mathbb{Z}}\left|\int_{\mathbb{R}} \overline{f(x)}e^{ib(\epsilon)\beta_{mn}(\epsilon)e^{-\epsilon x}}e^{-\frac{\epsilon}{2}(x+nq_0)}\psi(x+nq_0)dx \right|^2=\\ \nonumber
&&\sum_{n,m\in \mathbb{Z}}\left|\int_{\mathbb{R}^+} \overline{f(-\frac{\ln y}{\epsilon})}e^{ib(\epsilon)\beta_{mn}(\epsilon)y}e^{-\frac{\epsilon}{2}(-\frac{\ln y}{\epsilon}+nq_0)}\psi(-\frac{\ln y}{\epsilon}+nq_0)\frac{dy}{\epsilon y} \right|^2=\\ \nonumber
&&\sum_{n,m\in \mathbb{Z}}\left|\int_{\mathbb{R}^+} \overline{f(-\frac{\ln y}{\epsilon})}e^{ib(\epsilon)\beta_{mn}(\epsilon)y}e^{-\frac{\epsilon}{2}(nq_0)}\psi(-\frac{\ln y}{\epsilon}+nq_0)\frac{dy}{\epsilon \sqrt{y}} \right|^2=\\ \nonumber
&&\sum_{n,m\in \mathbb{Z}}\left|\int_{e^{-\epsilon L}}^{e^{\epsilon L}} \overline{f(-\frac{\ln (z\alpha^{-1}_n(\epsilon))}{\epsilon})}e^{ib(\epsilon)\beta_{mn}(\epsilon)\frac{z}{\alpha_n(\epsilon)}}e^{-\frac{\epsilon}{2}(nq_0)}\psi(-\frac{\ln z}{\epsilon})\frac{dz}{\epsilon \sqrt{\alpha_n(\epsilon) z}} \right|^2=\\ \nonumber
&&\sum_{n,m\in \mathbb{Z}}\left|\int_{e^{-\epsilon L}}^{e^{\epsilon L}} \overline{f(-\frac{\ln (z\alpha^{-1}_n(\epsilon))}{\epsilon})}e^{-\frac{2\pi    m}{e^{\epsilon L}-e^{-\epsilon L}}z}e^{-\frac{\epsilon}{2}(nq_0)}\psi(-\frac{\ln z}{\epsilon})\frac{dz}{\epsilon \sqrt{\alpha_n(\epsilon) z}} \right|^2
\end{eqnarray} 
Using Parseval's theorem we obtain that the above expression is equal to
\begin{eqnarray}\nonumber
&&2\sinh (\epsilon L) \sum_{n\in \mathbb{Z}}\int_{e^{-\epsilon L}}^{e^{\epsilon L}} \left| \overline{f(-\frac{\ln (z\alpha^{-1}_n(\epsilon))}{\epsilon})}\frac{e^{-\frac{\epsilon}{2}(nq_0)}\psi(-\frac{\ln z}{\epsilon})}{\epsilon \sqrt{\alpha_n(\epsilon)z}}\right|^2 dz\\ \nonumber
&&=2\frac{\sinh (\epsilon L)}{\epsilon} \sum_{n\in \mathbb{Z}}\int_{\mathbb{R}^+}\left| \overline{f(-\frac{\ln (z\alpha^{-1}_n(\epsilon))}{\epsilon})}\psi(-\frac{\ln z}{\epsilon})\right|^2\frac{1}{\epsilon z} dz
\end{eqnarray}
Applying the change of variables $y=\frac{z}{\alpha_n(\epsilon)}$  and then $x=-\frac{\ln y}{\epsilon}$ we obtain that the above expression is equal to
\begin{eqnarray}\nonumber
&&2\frac{\sinh (\epsilon L)}{\epsilon} \sum_{n\in \mathbb{Z}}\int_{\mathbb{R}^+} \left| \overline{f(-\frac{\ln y}{\epsilon})}\psi(-\frac{\ln ( y\alpha_n(\epsilon))}{\epsilon})\right|^2\frac{1}{\epsilon y} dy=\\ \nonumber
&&2\frac{\sinh (\epsilon L)}{\epsilon} \sum_{n\in \mathbb{Z}}\int_{\mathbb{R}^+} \left| \overline{f(-\frac{\ln y}{\epsilon})}\psi(-\frac{\ln y}{\epsilon}+nq_0)\right|^2\frac{1}{\epsilon y} dy=\\ \nonumber
&&2\frac{\sinh (\epsilon L)}{\epsilon} \sum_{n\in \mathbb{Z}}\int_{\mathbb{R}}\left| \overline{f(x)}\psi(x+nq_0)\right|^2dx=2\frac{\sinh (\epsilon L)}{\epsilon} \int_{\mathbb{R}}\left|{f(x)}\right|^2 \sum_{n\in \mathbb{Z}}\left|\psi(x+nq_0)\right|^2dx=\\ \nonumber
&& 2\chi \frac{\sinh (\epsilon L)}{\epsilon}\int_{\mathbb{R}} \left|f(x)\right|^2dx =2\chi \frac{\sinh (\epsilon L)}{\epsilon} \|f\|^2_{L^2(\mathbb{R},dx)}
\end{eqnarray}
The interchange of the order of the  sum and the integral  is justified by the compactness of the support of $\psi$.
\end{proof}

\begin{corollary}
\textbf{$[$Contraction of tight frames$]$} 
For any $f\in L^2(\mathbb{R},dx)$
\begin{eqnarray}\nonumber
&&\lim_{\epsilon \longrightarrow 0^+}\sum_{n,m\in \mathbb{Z}} \left|\langle f| \psi^{\epsilon,a(\epsilon),b(\epsilon)}_{(n,m)}\rangle\right|^2=\sum_{n,m\in \mathbb{Z}}\left|\langle f| \psi^{A,B}_{(n,m)}\rangle\right|^2
\end{eqnarray}
\end{corollary}
\begin{proof}
\begin{eqnarray}\nonumber
\lim_{\epsilon \longrightarrow 0^+}&&\sum_{n,m\in \mathbb{Z}} \left|\langle f| \psi^{\epsilon,a(\epsilon),b(\epsilon)}_{(n,m)}\rangle\right|^2=\lim_{\epsilon \longrightarrow 0^+}2\chi \frac{\sinh (\epsilon L)}{\epsilon} \|f\|^2_{L^2(\mathbb{R},dx)}=\\ \nonumber
&&2\chi L \|f\|^2_{L^2(\mathbb{R},dx)}= \sum_{n,m\in \mathbb{Z}}\left|\langle f| \psi^{A,B}_{(n,m)}\rangle\right|^2
\end{eqnarray}
\end{proof}
\begin{corollary}
\textbf{$[$Contraction of frame expansion$]$} 
For any $f\in L^2(\mathbb{R},dx)$ and any $\epsilon \in (0,1]$
\begin{eqnarray}\nonumber
f=&&\sum_{n,m\in \mathbb{Z}}\frac{\epsilon} {2\chi \sinh (\epsilon  L)} \left(\langle  \psi^{\epsilon,a(\epsilon),b(\epsilon)}_{(n,m)}|f\rangle \right) |\psi^{\epsilon,a(\epsilon),b(\epsilon)}_{(n,m)}\rangle=\sum_{n,m\in \mathbb{Z}}\frac{1} {2\chi L} \left(\langle  \psi^{A,B}_{(n,m)}|f\rangle \right) |\psi^{A,B}_{(n,m)}\rangle
\end{eqnarray}
where the equality is in $L^2(\mathbb{R},dx)$.\label{cor6.2}
\end{corollary}

\begin{proposition}
\textbf{$[$Contraction of  overcomplete frame bases$]$} 
 For any $x\in \mathbb{R}$, $(q,p)\in \mathbb{R}^2$
\begin{eqnarray}
&&\lim_{\epsilon \longrightarrow 0^+}|\psi^{\epsilon,a(\epsilon),b(\epsilon)}_{(n,m)} \rangle(x) =| \psi^{A,B}_{(n,m)} \rangle(x)
\end{eqnarray}
and
\begin{eqnarray}\label{5.31}
&&\lim_{\epsilon \longrightarrow 0^+}|\psi^{\epsilon,a(\epsilon),b(\epsilon)}_{(n,m)} \rangle =| \psi^{A,B}_{(n,m)} \rangle
\end{eqnarray}
i.e., the convergence is pointwise and in  norm.
\label{propIV5}
\end{proposition}
\begin{proof}
We observe that
\begin{eqnarray}\nonumber
&&|\psi^{\epsilon,a(\epsilon),b(\epsilon)}_{(n,m)} \rangle(x) =\rho^{a(\epsilon),b(\epsilon)}_{\epsilon}(\alpha_n(\epsilon), \beta_{mn}(\epsilon), \gamma_{mn}(\epsilon))\left( Q_{\epsilon}\psi \right)(x)\\ \nonumber
&&=e^{ia(\epsilon)\gamma_{mn}(\epsilon)}e^{ib(\epsilon)\beta_{mn}(\epsilon)e^{-\epsilon x}}e^{-\frac{\epsilon}{2}(x+nq_0)}\psi(x+nq_0)=
\end{eqnarray}
\begin{widetext}
\begin{eqnarray}\nonumber
&&\exp \left\{ i\beta_{mn}(\epsilon)\left(a(\epsilon)\frac{\ln \alpha_n(\epsilon)}{\alpha_n(\epsilon)-1}+b(\epsilon)e^{-\epsilon x}   \right)\right\}=\\ \nonumber
&& \exp \left\{ ie^{-\epsilon nq_0}\frac{-2\pi    m}{(-\frac{A}{\epsilon}+b_0)(e^{\epsilon L}-e^{-\epsilon L})}\left((\frac{A}{\epsilon}+a_0)\frac{-\epsilon nq_0}{e^{-\epsilon nq_0}-1}+(-\frac{A}{\epsilon}+b_0)e^{-\epsilon x}   \right)\right\}
\end{eqnarray}
\end{widetext}
and since
\begin{eqnarray}\nonumber
&& \lim_{\epsilon \longrightarrow 0^+}(-\frac{A}{\epsilon}+b_0)(e^{\epsilon L}-e^{-\epsilon L})=\lim_{\epsilon \longrightarrow 0^+}(-\frac{A}{\epsilon}+b_0)(2\epsilon L+o(\epsilon))=-2AL
\end{eqnarray}
then 
\begin{eqnarray}\nonumber
&& \lim_{\epsilon \longrightarrow 0^+} \left((\frac{A}{\epsilon}+a_0)\frac{-\epsilon nq_0}{e^{-\epsilon nq_0}-1}+(-\frac{A}{\epsilon}+b_0)e^{-\epsilon x}   \right)=\\ \nonumber
&& \lim_{\epsilon \longrightarrow 0^+} \left((\frac{A}{\epsilon}+a_0)(1-\frac{\epsilon nq_0}{2})+(-\frac{A}{\epsilon}+b_0)(1-\epsilon x)   \right)\\ \nonumber
&& =a_0+b_0+\frac{A nq_0}{2}+Ax=B+\frac{A nq_0}{2}+Ax
\end{eqnarray}
and
\begin{eqnarray}\nonumber
&&\lim_{\epsilon \longrightarrow 0^+}|\psi^{\epsilon,a(\epsilon),b(\epsilon)}_{(n,m)} \rangle(x)= e^{i\frac{\pi m}{AL}(B+\frac{A nq_0}{2}+Ax)} \psi(x+nq_0)= e^{ip_0m(B+\frac{A nq_0}{2}+Ax)} \psi(x+nq_0)=| \psi^{A,B}_{(n,m)} \rangle(x)
\end{eqnarray}
Equation \ref{5.31} follows from the dominated convergence theorem.
\end{proof}
\begin{corollary}
\textbf{$[$Contraction of truncated frame expansions$]$} 
for any $f\in L^2(\mathbb{R},dx)$ and $N_1,N_2,M_1,M_2\in \mathbb{Z}$
\begin{eqnarray}\label{5.37}\nonumber
\lim_{\epsilon \longrightarrow 0^+}&&\sum_{n=N_1}^{n=N_2}\sum_{m=M_1}^{m=M_2}\frac{\epsilon} {2\chi \sinh (\epsilon  L)} \left(\langle  \psi^{\epsilon,a(\epsilon),b(\epsilon)}_{(n,m)}|f\rangle \right) |\psi^{\epsilon,a(\epsilon),b(\epsilon)}_{(n,m)}\rangle =\sum_{n=N_1}^{n=N_2}\sum_{m=M_1}^{m=M_2}\frac{1} {2\chi L} \left(\langle  \psi^{A,B}_{(n,m)}|f\rangle \right) |\psi^{A,B}_{(n,m)}\rangle
\end{eqnarray}
where the equality is in $L^2(\mathbb{R},dx)$ and pointwise.
\end{corollary}

\section{Frames with band limited fiducial states and their  contractions}\label{sec6}
\noindent In this section we construct another realization of the family of representations $\left\{\rho^{a(\epsilon),b(\epsilon)}_{\epsilon}\right\}_{\epsilon\in (0,1]}$
which also contracts to the desired representation of the Heisenberg group. This realization enables us to construct families of tight frames  based on functions
with compactly supported  Fourier transform. Let us recall how we constructed the representation  $\rho^{a,b}_{\epsilon}:EA\longrightarrow \mathcal{U}(L^2(\mathbb{R},dx))$. 
In section \ref{secb} we have defined 
\begin{eqnarray}\nonumber
&& \rho_{\epsilon}^{a,b}:EA \longrightarrow \mathcal{U}(V_{\epsilon})\\ \nonumber
&& (\alpha, \beta, \gamma) \longmapsto T_{\epsilon} \circ (U^b\otimes \chi_{a})(\alpha, \beta, \gamma) \circ T^{-1}_{\epsilon}
\end{eqnarray}
where $T_{\epsilon}=S_1(\epsilon)\circ S_2 \circ S_1$ goes from   $\mathcal{H}=\left\{f\in L^2(\mathbb{R},dx)|\mathcal{F}(f)(x)=0, \forall x <0 \right\}$ onto $L^2(\mathbb{R},dx)$. 
 Let $I$ denote the unitary operator on $L^2(\mathbb{R})$
that is given by   $T^{-1}_{\epsilon=1}=S_1^{-1}\circ S_2^{-1}\circ S^{-1}_3(1)$. We define $\pi_{\epsilon}^{a,b}:EA \longrightarrow GL(\mathcal{H})$ to be the representation that is given by $I\circ \rho_{\epsilon}^{a,b}\circ I^{-1}$; explicitly we have
\begin{eqnarray}\nonumber
& \pi_{\epsilon}^{a,b}&\left( \alpha, \beta, \gamma \right)(f)(x)=e^{ia\gamma}\alpha^{\frac{1}{2\epsilon}}
\mathcal{F}^{-1}\left(e^{ib\beta w^{\epsilon}}
\mathcal{F}(f)(\alpha^{\frac{1}{\epsilon}}w)\right)(x)\\ \nonumber
&&=\frac{e^{ia\gamma}\alpha^{\frac{1}{2\epsilon}}}{{2\pi}}\int_{\mathbb{R}}\int_{\mathbb{R}} f(t)e^{ib\beta w^{\epsilon}-i\alpha^{\frac{1}{\epsilon}}wt+iwx} dtdw
\end{eqnarray}
Using the known representation of the delta distribution, $\frac{1}{2\pi}\int_{\mathbb{R}}e^{ik(x-y)}dk=\delta(x-y)$
 we verify that $\pi_{\epsilon =1}^{a,b}=U^b\otimes \chi_{a}$ as follows
\begin{eqnarray}\nonumber
&&\pi_{1}^{a,b}\left( (\alpha, \beta, \gamma) \right)(f)(x)=\frac{e^{ia\gamma}\alpha^{\frac{1}{2}}}{{2\pi}}\int_{\mathbb{R}}\int_{\mathbb{R}} f(t)e^{ib\beta w-i\alpha wt+iwx} dtdw=\\ \nonumber
&&\frac{e^{ia\gamma}\alpha^{\frac{1}{2}}}{{2\pi}}\int_{\mathbb{R}}\int_{\mathbb{R}} f(t)e^{iw(b\beta -\alpha t+x)} dwdt=e^{ia\gamma}\alpha^{\frac{1}{2}}\int_{\mathbb{R}} f(t)\delta(b\beta -\alpha t+x)  dt=\\ \nonumber
&& e^{ia \gamma}\frac{1}{\sqrt{\alpha}}f(\frac{x+b\beta}{\alpha})=(U^b\otimes \chi_{a}((\alpha, \beta, \gamma))f)(x)
\end{eqnarray}
Let $\widetilde{\eta}^{A,B}$ denote the representation of $H$ that is given by $I \circ \eta^{A,B}\circ I^{-1}$.
A direct calculation shows that for any $\left(c,\left(
                                                 \begin{array}{c}
                                                   v_1 \\
                                                   v_2 \\
                                                 \end{array}
                                               \right)
\right) \in H$
\begin{eqnarray}\nonumber
&&\left(\widetilde{\eta}^{A,B}\left(c,\left(
                                                 \begin{array}{c}
                                                   v_1 \\
                                                   v_2 \\
                                                 \end{array}
                                               \right) \right)f\right)(x)=e^{i(Av_2+Bv_1)}f\left(-\log(c+e^{Av_2-x})\right)
\end{eqnarray}
\noindent The unitarity of $I$  allows us to translate statements regarding the family $ \left\{\rho_{\epsilon}^{a(\epsilon),b(\epsilon)}\right\}_{\epsilon \in (0,1]} $
and its contraction $\eta^{A,B}$ to the family $\left\{\pi_{\epsilon}^{a(\epsilon),b(\epsilon)}\right\}_{\epsilon \in (0,1]} $ and its contraction $\widetilde{\eta}^{A,B}$.
Here we only translate the statement with regards to the tight frame:
\begin{proposition}
\textbf{$[$band limited tight frames for $H$$]$}
Let $0\neq\psi \in L^2(\mathbb{R},dx)$ satisfy the assumptions of Theorem \ref{pr1}.
Then the sequence $\left\{\left| \widehat{\psi}^{A,B} _{(n,m)}\right\rangle\right\}_{n,m\in \mathbb{Z}}$, where
\begin{eqnarray}\nonumber
&&\left| \widehat{\psi}^{A,B} _{(n,m)}\right\rangle= \widetilde{\eta}^{A,B}\left(nq_0, \left(\begin{array}{c}
\frac{nq_0mp_0}{2}\\
mp_0\end{array}\right)\right)  (I\psi)
\end{eqnarray}
constitutes a tight frame with frame constant that is equal to $2L\chi $.
\label{prop9}
\end{proposition}

\begin{proposition}
\textbf{$[$band limited tight frames for $EA$$]$} 
Under the assumptions of Theorem \ref{pr1},  for any $\epsilon \in (0,1] $ let $EA_{q_0,p_0}(\epsilon)$ be the
discrete subset of  $EA$ that was defined in Proposition \ref{propo7} by
$\left\{ \left(\alpha_n(\epsilon), \beta_{mn}(\epsilon), \gamma_{mn}(\epsilon) \right) |m,n \in \mathbb{Z}\right\}$. Let $\psi_{\epsilon}=IQ_{\epsilon}\psi$.
The sequence $\left\{ \left|\widehat{\psi}^{\epsilon,a(\epsilon),b(\epsilon)}_{(n,m)} \right\rangle \right\}_{m,n \in \mathbb{Z}}$, where $\left|\widehat{\psi}^{\epsilon,a(\epsilon),b(\epsilon)}_{(n,m)} \right\rangle=
\pi^{a(\epsilon),b(\epsilon)}_{\epsilon}(\alpha_n(\epsilon), \beta_{mn}(\epsilon), \gamma_{mn}(\epsilon))\left( \psi_{\epsilon} \right)$
constitutes a tight frame with frame constant that is equal to $2\frac{\sinh (\epsilon L)}{\epsilon}\chi$.
\label{prop10}
\end{proposition}
\begin{remark}
The tight frames given in Propositions \ref{prop9} and \ref{prop10} are composed of functions whose
Fourier transform is compactly supported, and hence suitable for analyzing band limited  signals.
\end{remark}

\begin{corollary}
\textbf{$[$Contraction of tight frames$]$}
For any $f\in L^2(\mathbb{R},dx)$
\begin{eqnarray}\nonumber
&&\lim_{\epsilon \longrightarrow 0^+}\sum_{n,m\in \mathbb{Z}} \left|\langle f| \widehat{\psi}^{\epsilon,a(\epsilon),b(\epsilon)}_{(n,m)}\rangle\right|^2=\sum_{n,m\in \mathbb{Z}}\left|\langle f| \widehat{\psi}^{A,B}_{(n,m)}\rangle\right|^2
\end{eqnarray}
\end{corollary}
\section{Concluding remarks}\label{sec7}
\noindent It should be noted that one can not directly relate  the affine group to the Heisenberg group by contraction, since these groups are not of the same dimension.
An essential step in our approach was to consider the extended affine group which gives rise to the same analysis as the affine group and to use the contraction
of the extended affine group to the Heisenberg group.

In practice, when one analyzes signals by the standard wavelets or Gabor frames there is a hard problem of choosing the most suitable frame.
The frames that are given
 in this paper come in natural families that vary continuously in the parameter $\epsilon$.
This extra parameter, which is absent in the usual theory,  gives rise to a setting in which  one can potentially  choose a more suitable frame. This approach for \textit{synchronization} over Cartan motion groups was proven to be very efficient in a recent work \cite{Sharon2016}. 

For band limited signals our family of frames coincides for $\epsilon =1$ with the standard wavelet frame. For $\epsilon =0$ we obtain a new frame of Gabor type.
 This gives us a continuous family of frames,
which can be regarded as a continuous analog  of a discrete library (for an example of a discrete library see \cite{Coif}). Given a signal, it is a basic question
 to find an $\epsilon$ in the range $[0,1]$ that gives the best approximating frame. 
 
 This paper serves as a first step in the study of deformation of Gabor and wavelets frames using group contractions. We believe that the results presented here suggest further study in this direction.

\bibliography{references}

\def\cprime{$'$}
\begin{thebibliography}{33}%
\makeatletter
\providecommand \@ifxundefined [1]{%
 \@ifx{#1\undefined}
}%
\providecommand \@ifnum [1]{%
 \ifnum #1\expandafter \@firstoftwo
 \else \expandafter \@secondoftwo
 \fi
}%
\providecommand \@ifx [1]{%
 \ifx #1\expandafter \@firstoftwo
 \else \expandafter \@secondoftwo
 \fi
}%
\providecommand \natexlab [1]{#1}%
\providecommand \enquote  [1]{``#1''}%
\providecommand \bibnamefont  [1]{#1}%
\providecommand \bibfnamefont [1]{#1}%
\providecommand \citenamefont [1]{#1}%
\providecommand \href@noop [0]{\@secondoftwo}%
\providecommand \href [0]{\begingroup \@sanitize@url \@href}%
\providecommand \@href[1]{\@@startlink{#1}\@@href}%
\providecommand \@@href[1]{\endgroup#1\@@endlink}%
\providecommand \@sanitize@url [0]{\catcode `\\12\catcode `\$12\catcode
  `\&12\catcode `\#12\catcode `\^12\catcode `\_12\catcode `\%12\relax}%
\providecommand \@@startlink[1]{}%
\providecommand \@@endlink[0]{}%
\providecommand \url  [0]{\begingroup\@sanitize@url \@url }%
\providecommand \@url [1]{\endgroup\@href {#1}{\urlprefix }}%
\providecommand \urlprefix  [0]{URL }%
\providecommand \Eprint [0]{\href }%
\providecommand \doibase [0]{http://dx.doi.org/}%
\providecommand \selectlanguage [0]{\@gobble}%
\providecommand \bibinfo  [0]{\@secondoftwo}%
\providecommand \bibfield  [0]{\@secondoftwo}%
\providecommand \translation [1]{[#1]}%
\providecommand \BibitemOpen [0]{}%
\providecommand \bibitemStop [0]{}%
\providecommand \bibitemNoStop [0]{.\EOS\space}%
\providecommand \EOS [0]{\spacefactor3000\relax}%
\providecommand \BibitemShut  [1]{\csname bibitem#1\endcsname}%
\let\auto@bib@innerbib\@empty
%</preamble>
\bibitem [{\citenamefont {Ali}\ \emph {et~al.}(2014)\citenamefont {Ali},
  \citenamefont {Antoine},\ and\ \citenamefont {Gazeau}}]{ali2013coherent}%
  \BibitemOpen
  \bibfield  {author} {\bibinfo {author} {\bibfnamefont {S.}~\bibnamefont
  {Ali}}, \bibinfo {author} {\bibfnamefont {J.}~\bibnamefont {Antoine}}, \ and\
  \bibinfo {author} {\bibfnamefont {J.}~\bibnamefont {Gazeau}},\ }\href
  {https://books.google.com/books?id=gBy4BAAAQBAJ} {\emph {\bibinfo {title}
  {Coherent States, Wavelets, and Their Generalizations}}},\ Theoretical and
  Mathematical Physics\ (\bibinfo  {publisher} {Springer, New York},\ \bibinfo
  {year} {2014})\BibitemShut {NoStop}%
\bibitem [{\citenamefont {Christensen}(2016)}]{Chris2016}%
  \BibitemOpen
  \bibfield  {author} {\bibinfo {author} {\bibfnamefont {O.}~\bibnamefont
  {Christensen}},\ }\href {\doibase 10.1007/978-3-319-25613-9} {\emph {\bibinfo
  {title} {An introduction to frames and {R}iesz bases}}},\ \bibinfo {edition}
  {2nd}\ ed.,\ Applied and Numerical Harmonic Analysis\ (\bibinfo  {publisher}
  {Birkh\"auser/Springer, [Cham]},\ \bibinfo {year} {2016})\ pp.\ \bibinfo
  {pages} {xxv+704}\BibitemShut {NoStop}%
\bibitem [{\citenamefont {Feichtinger}\ and\ \citenamefont
  {Strohmer}(2003)}]{Feichtinger2003}%
  \BibitemOpen
  \bibinfo {editor} {\bibfnamefont {H.~G.}\ \bibnamefont {Feichtinger}}\ and\
  \bibinfo {editor} {\bibfnamefont {T.}~\bibnamefont {Strohmer}},\ eds.,\ \href
  {\doibase 10.1007/978-1-4612-0133-5} {\emph {\bibinfo {title} {Advances in
  {G}abor analysis}}},\ Applied and Numerical Harmonic Analysis\ (\bibinfo
  {publisher} {Birkh\"auser Boston, Inc., Boston, MA},\ \bibinfo {year}
  {2003})\ pp.\ \bibinfo {pages} {xx+356}\BibitemShut {NoStop}%
\bibitem [{\citenamefont {Feichtinger}\ and\ \citenamefont
  {Strohmer}(1998)}]{Feichtinger1998}%
  \BibitemOpen
  \bibinfo {editor} {\bibfnamefont {H.~G.}\ \bibnamefont {Feichtinger}}\ and\
  \bibinfo {editor} {\bibfnamefont {T.}~\bibnamefont {Strohmer}},\ eds.,\ \href
  {\doibase 10.1007/978-1-4612-2016-9} {\emph {\bibinfo {title} {Gabor analysis
  and algorithms}}},\ Applied and Numerical Harmonic Analysis\ (\bibinfo
  {publisher} {Birkh\"auser Boston, Inc., Boston, MA},\ \bibinfo {year}
  {1998})\ pp.\ \bibinfo {pages} {xvi+496},\ \bibinfo {note} {theory and
  applications}\BibitemShut {NoStop}%
\bibitem [{\citenamefont {Gr\"ochenig}(2001)}]{MR1843717}%
  \BibitemOpen
  \bibfield  {author} {\bibinfo {author} {\bibfnamefont {K.}~\bibnamefont
  {Gr\"ochenig}},\ }\href {\doibase 10.1007/978-1-4612-0003-1} {\emph {\bibinfo
  {title} {Foundations of time-frequency analysis}}},\ Applied and Numerical
  Harmonic Analysis\ (\bibinfo  {publisher} {Birkh\"auser Boston, Inc., Boston,
  MA},\ \bibinfo {year} {2001})\ pp.\ \bibinfo {pages} {xvi+359}\BibitemShut
  {NoStop}%
\bibitem [{\citenamefont {Perelomov}(1986)}]{perelomov2012generalized}%
  \BibitemOpen
  \bibfield  {author} {\bibinfo {author} {\bibfnamefont {A.}~\bibnamefont
  {Perelomov}},\ }\href@noop {} {\emph {\bibinfo {title} {Generalized Coherent
  States and Their Applications}}}\ (\bibinfo  {publisher} {Springer-Verlag,
  Berlin},\ \bibinfo {year} {1986})\BibitemShut {NoStop}%
\bibitem [{\citenamefont {Inonu}\ and\ \citenamefont
  {Wigner}(1953)}]{Inonu-Wigner53}%
  \BibitemOpen
  \bibfield  {author} {\bibinfo {author} {\bibfnamefont {E.}~\bibnamefont
  {Inonu}}\ and\ \bibinfo {author} {\bibfnamefont {E.~P.}\ \bibnamefont
  {Wigner}},\ }\href@noop {} {\bibfield  {journal} {\bibinfo  {journal} {Proc.
  Nat. Acad. Sci. U. S. A.}\ }\textbf {\bibinfo {volume} {39}},\ \bibinfo
  {pages} {510} (\bibinfo {year} {1953})}\BibitemShut {NoStop}%
\bibitem [{\citenamefont {Gilmore}(1974)}]{gilmore2012lie}%
  \BibitemOpen
  \bibfield  {author} {\bibinfo {author} {\bibfnamefont {R.}~\bibnamefont
  {Gilmore}},\ }\href@noop {} {\emph {\bibinfo {title} {Lie Groups, Lie
  Algebras, and Some of Their Applications}}}\ (\bibinfo  {publisher} {John
  Wiley and Sons, New York},\ \bibinfo {year} {1974})\BibitemShut {NoStop}%
\bibitem [{\citenamefont {Subag}\ \emph {et~al.}(2012)\citenamefont {Subag},
  \citenamefont {Baruch}, \citenamefont {Birman},\ and\ \citenamefont
  {Mann}}]{Subag12}%
  \BibitemOpen
  \bibfield  {author} {\bibinfo {author} {\bibfnamefont {E.~M.}\ \bibnamefont
  {Subag}}, \bibinfo {author} {\bibfnamefont {E.~M.}\ \bibnamefont {Baruch}},
  \bibinfo {author} {\bibfnamefont {J.~L.}\ \bibnamefont {Birman}}, \ and\
  \bibinfo {author} {\bibfnamefont {A.}~\bibnamefont {Mann}},\ }\href {\doibase
  10.1088/1751-8113/45/26/265206} {\bibfield  {journal} {\bibinfo  {journal}
  {J. Phys. A}\ }\textbf {\bibinfo {volume} {45}},\ \bibinfo {pages} {265206}
  (\bibinfo {year} {2012})}\BibitemShut {NoStop}%
\bibitem [{\citenamefont {Torresani}(1992)}]{Tor1}%
  \BibitemOpen
  \bibfield  {author} {\bibinfo {author} {\bibfnamefont {B.}~\bibnamefont
  {Torresani}},\ }\href@noop {} {\bibfield  {journal} {\bibinfo  {journal}
  {Annales de l'I.H.P., section A}\ }\textbf {\bibinfo {volume} {56}},\
  \bibinfo {pages} {215} (\bibinfo {year} {1992})}\BibitemShut {NoStop}%
\bibitem [{\citenamefont {Torresani}(1991)}]{Tor2}%
  \BibitemOpen
  \bibfield  {author} {\bibinfo {author} {\bibfnamefont {B.}~\bibnamefont
  {Torresani}},\ }\href@noop {} {\bibfield  {journal} {\bibinfo  {journal} {J.
  Math. Phys.}\ }\textbf {\bibinfo {volume} {32}},\ \bibinfo {pages} {1273}
  (\bibinfo {year} {1991})}\BibitemShut {NoStop}%
\bibitem [{\citenamefont {Dahlke}\ \emph {et~al.}(2008)\citenamefont {Dahlke},
  \citenamefont {Fornasier}, \citenamefont {Rauhut}, \citenamefont {Steidl},\
  and\ \citenamefont {Teschke}}]{Dahlke2008}%
  \BibitemOpen
  \bibfield  {author} {\bibinfo {author} {\bibfnamefont {S.}~\bibnamefont
  {Dahlke}}, \bibinfo {author} {\bibfnamefont {M.}~\bibnamefont {Fornasier}},
  \bibinfo {author} {\bibfnamefont {H.}~\bibnamefont {Rauhut}}, \bibinfo
  {author} {\bibfnamefont {G.}~\bibnamefont {Steidl}}, \ and\ \bibinfo {author}
  {\bibfnamefont {G.}~\bibnamefont {Teschke}},\ }\href {\doibase
  10.1112/plms/pdm051} {\bibfield  {journal} {\bibinfo  {journal} {Proc. Lond.
  Math. Soc. (3)}\ }\textbf {\bibinfo {volume} {96}},\ \bibinfo {pages} {464}
  (\bibinfo {year} {2008})}\BibitemShut {NoStop}%
\bibitem [{\citenamefont {Hogan}\ and\ \citenamefont
  {Lakey}(1995)}]{MR1325539}%
  \BibitemOpen
  \bibfield  {author} {\bibinfo {author} {\bibfnamefont {J.~A.}\ \bibnamefont
  {Hogan}}\ and\ \bibinfo {author} {\bibfnamefont {J.~D.}\ \bibnamefont
  {Lakey}},\ }\href {\doibase 10.1006/acha.1995.1013} {\bibfield  {journal}
  {\bibinfo  {journal} {Appl. Comput. Harmon. Anal.}\ }\textbf {\bibinfo
  {volume} {2}},\ \bibinfo {pages} {174} (\bibinfo {year} {1995})}\BibitemShut
  {NoStop}%
\bibitem [{\citenamefont {Kalnins}\ and\ \citenamefont
  {Miller}(1992)}]{Willard}%
  \BibitemOpen
  \bibfield  {author} {\bibinfo {author} {\bibfnamefont {E.~G.}\ \bibnamefont
  {Kalnins}}\ and\ \bibinfo {author} {\bibfnamefont {W.~J.}\ \bibnamefont
  {Miller}},\ }in\ \href@noop {} {\emph {\bibinfo {booktitle} {Radar and sonar,
  Part II}}},\ \bibinfo {series} {IMA Vol. Math. Appl.}, Vol.~\bibinfo {volume}
  {39}\ (\bibinfo  {publisher} {Springer, New York},\ \bibinfo {year} {1992})\
  pp.\ \bibinfo {pages} {183--196}\BibitemShut {NoStop}%
\bibitem [{\citenamefont {Segal}(1951)}]{Segal}%
  \BibitemOpen
  \bibfield  {author} {\bibinfo {author} {\bibfnamefont {I.~E.}\ \bibnamefont
  {Segal}},\ }\href@noop {} {\bibfield  {journal} {\bibinfo  {journal} {Duke
  Math. J.}\ }\textbf {\bibinfo {volume} {18}},\ \bibinfo {pages} {221}
  (\bibinfo {year} {1951})}\BibitemShut {NoStop}%
\bibitem [{\citenamefont {Saletan}(1961)}]{Saletan61}%
  \BibitemOpen
  \bibfield  {author} {\bibinfo {author} {\bibfnamefont {E.~J.}\ \bibnamefont
  {Saletan}},\ }\href@noop {} {\bibfield  {journal} {\bibinfo  {journal} {J.
  Mathematical Phys.}\ }\textbf {\bibinfo {volume} {2}},\ \bibinfo {pages} {1}
  (\bibinfo {year} {1961})}\BibitemShut {NoStop}%
\bibitem [{\citenamefont {Mickelsson}\ and\ \citenamefont
  {Niederle}(1972)}]{Mic}%
  \BibitemOpen
  \bibfield  {author} {\bibinfo {author} {\bibfnamefont {J.}~\bibnamefont
  {Mickelsson}}\ and\ \bibinfo {author} {\bibfnamefont {J.}~\bibnamefont
  {Niederle}},\ }\href@noop {} {\bibfield  {journal} {\bibinfo  {journal}
  {Commun. math. Phys}\ }\textbf {\bibinfo {volume} {27}},\ \bibinfo {pages}
  {167} (\bibinfo {year} {1972})}\BibitemShut {NoStop}%
\bibitem [{\citenamefont {Dooley}\ and\ \citenamefont {Rice}(1985)}]{Dooley85}%
  \BibitemOpen
  \bibfield  {author} {\bibinfo {author} {\bibfnamefont {A.~H.}\ \bibnamefont
  {Dooley}}\ and\ \bibinfo {author} {\bibfnamefont {J.~W.}\ \bibnamefont
  {Rice}},\ }\href {\doibase 10.2307/1999695} {\bibfield  {journal} {\bibinfo
  {journal} {Trans. Amer. Math. Soc.}\ }\textbf {\bibinfo {volume} {289}},\
  \bibinfo {pages} {185} (\bibinfo {year} {1985})}\BibitemShut {NoStop}%
\bibitem [{\citenamefont {Ricci}(1986)}]{Ricci}%
  \BibitemOpen
  \bibfield  {author} {\bibinfo {author} {\bibfnamefont {F.}~\bibnamefont
  {Ricci}},\ }\href@noop {} {\bibfield  {journal} {\bibinfo  {journal} {Mh.
  Math.}\ }\textbf {\bibinfo {volume} {101}},\ \bibinfo {pages} {211} (\bibinfo
  {year} {1986})}\BibitemShut {NoStop}%
\bibitem [{\citenamefont {Daubechies}\ \emph {et~al.}(1986)\citenamefont
  {Daubechies}, \citenamefont {Grossmann},\ and\ \citenamefont {Meyer}}]{Dau}%
  \BibitemOpen
  \bibfield  {author} {\bibinfo {author} {\bibfnamefont {I.}~\bibnamefont
  {Daubechies}}, \bibinfo {author} {\bibfnamefont {A.}~\bibnamefont
  {Grossmann}}, \ and\ \bibinfo {author} {\bibfnamefont {Y.}~\bibnamefont
  {Meyer}},\ }\href {\doibase 10.1063/1.527388} {\bibfield  {journal} {\bibinfo
   {journal} {Journal of Mathematical Physics}\ }\textbf {\bibinfo {volume}
  {27}},\ \bibinfo {pages} {1271} (\bibinfo {year} {1986})}\BibitemShut
  {NoStop}%
\bibitem [{\citenamefont {Renaud}(1996)}]{Renaud}%
  \BibitemOpen
  \bibfield  {author} {\bibinfo {author} {\bibfnamefont {J.}~\bibnamefont
  {Renaud}},\ }\href {\doibase 10.1063/1.531563} {\bibfield  {journal}
  {\bibinfo  {journal} {Journal of Mathematical Physics}\ }\textbf {\bibinfo
  {volume} {37}},\ \bibinfo {pages} {3168} (\bibinfo {year}
  {1996})}\BibitemShut {NoStop}%
\bibitem [{\citenamefont {Christensen}\ and\ \citenamefont {Goh}(2014)}]{Ole}%
  \BibitemOpen
  \bibfield  {author} {\bibinfo {author} {\bibfnamefont {O.}~\bibnamefont
  {Christensen}}\ and\ \bibinfo {author} {\bibfnamefont {S.~S.}\ \bibnamefont
  {Goh}},\ }\href@noop {} {\bibfield  {journal} {\bibinfo  {journal} {Appl.
  Comput. Harmon. Anal.}\ }\textbf {\bibinfo {volume} {36}},\ \bibinfo {pages}
  {198} (\bibinfo {year} {2014})}\BibitemShut {NoStop}%
\bibitem [{\citenamefont {Feichtinger}\ and\ \citenamefont
  {Kaiblinger}(2004)}]{Feichtinger}%
  \BibitemOpen
  \bibfield  {author} {\bibinfo {author} {\bibfnamefont {H.~G.}\ \bibnamefont
  {Feichtinger}}\ and\ \bibinfo {author} {\bibfnamefont {N.}~\bibnamefont
  {Kaiblinger}},\ }\href@noop {} {\bibfield  {journal} {\bibinfo  {journal}
  {Trans. Amer. Math. Soc.}\ }\textbf {\bibinfo {volume} {356}},\ \bibinfo
  {pages} {2001} (\bibinfo {year} {2004})}\BibitemShut {NoStop}%
\bibitem [{\citenamefont {de~Gosson}(2015)}]{deGosson2015196}%
  \BibitemOpen
  \bibfield  {author} {\bibinfo {author} {\bibfnamefont {M.~A.}\ \bibnamefont
  {de~Gosson}},\ }\href {\doibase http://dx.doi.org/10.1016/j.acha.2014.03.010}
  {\bibfield  {journal} {\bibinfo  {journal} {Applied and Computational
  Harmonic Analysis}\ }\textbf {\bibinfo {volume} {38}},\ \bibinfo {pages} {196
  } (\bibinfo {year} {2015})}\BibitemShut {NoStop}%
\bibitem [{\citenamefont {Grochenig}\ \emph {et~al.}(2015)\citenamefont
  {Grochenig}, \citenamefont {Ortega-Cerda},\ and\ \citenamefont
  {Romero}}]{Grochenig}%
  \BibitemOpen
  \bibfield  {author} {\bibinfo {author} {\bibfnamefont {K.}~\bibnamefont
  {Grochenig}}, \bibinfo {author} {\bibfnamefont {J.}~\bibnamefont
  {Ortega-Cerda}}, \ and\ \bibinfo {author} {\bibfnamefont {J.~L.}\
  \bibnamefont {Romero}},\ }\href {\doibase
  http://dx.doi.org/10.1016/j.aim.2015.01.019} {\bibfield  {journal} {\bibinfo
  {journal} {Advances in Mathematics}\ }\textbf {\bibinfo {volume} {277}},\
  \bibinfo {pages} {388 } (\bibinfo {year} {2015})}\BibitemShut {NoStop}%
\bibitem [{\citenamefont {Mackey}(1949)}]{Mac}%
  \BibitemOpen
  \bibfield  {author} {\bibinfo {author} {\bibfnamefont {G.~W.}\ \bibnamefont
  {Mackey}},\ }\href@noop {} {\bibfield  {journal} {\bibinfo  {journal} {Proc.
  Nat. Acad. Sci. U.S}\ }\textbf {\bibinfo {volume} {35}},\ \bibinfo {pages}
  {537} (\bibinfo {year} {1949})}\BibitemShut {NoStop}%
\bibitem [{\citenamefont {von Neumann}(1955)}]{Von}%
  \BibitemOpen
  \bibfield  {author} {\bibinfo {author} {\bibfnamefont {J.}~\bibnamefont {von
  Neumann}},\ }\href {https://books.google.com/books?id=ePMX38H4DD4C} {\emph
  {\bibinfo {title} {Mathematical Foundations of Quantum Mechanics}}},\ Dover
  Books on Mathematics\ (\bibinfo  {publisher} {Princeton U. P., Princeton,
  NJ},\ \bibinfo {year} {1955})\BibitemShut {NoStop}%
\bibitem [{\citenamefont {Bargmann}\ \emph {et~al.}(1971)\citenamefont
  {Bargmann}, \citenamefont {Butera}, \citenamefont {Girardello},\ and\
  \citenamefont {Klauder}}]{Barg}%
  \BibitemOpen
  \bibfield  {author} {\bibinfo {author} {\bibfnamefont {V.}~\bibnamefont
  {Bargmann}}, \bibinfo {author} {\bibfnamefont {P.}~\bibnamefont {Butera}},
  \bibinfo {author} {\bibfnamefont {L.}~\bibnamefont {Girardello}}, \ and\
  \bibinfo {author} {\bibfnamefont {J.~R.}\ \bibnamefont {Klauder}},\
  }\href@noop {} {\bibfield  {journal} {\bibinfo  {journal} {Rep. Math. Phys}\
  }\textbf {\bibinfo {volume} {2}},\ \bibinfo {pages} {221} (\bibinfo {year}
  {1971})}\BibitemShut {NoStop}%
\bibitem [{\citenamefont {Perelomov}(1971)}]{Per2}%
  \BibitemOpen
  \bibfield  {author} {\bibinfo {author} {\bibfnamefont {A.~M.}\ \bibnamefont
  {Perelomov}},\ }\href@noop {} {\bibfield  {journal} {\bibinfo  {journal}
  {Teor. Mat. Fiz.}\ }\textbf {\bibinfo {volume} {6}},\ \bibinfo {pages} {213}
  (\bibinfo {year} {1971})}\BibitemShut {NoStop}%
\bibitem [{\citenamefont {Bacry}\ \emph {et~al.}(1975)\citenamefont {Bacry},
  \citenamefont {Grossmann},\ and\ \citenamefont {Zak}}]{Bac}%
  \BibitemOpen
  \bibfield  {author} {\bibinfo {author} {\bibfnamefont {H.}~\bibnamefont
  {Bacry}}, \bibinfo {author} {\bibfnamefont {A.}~\bibnamefont {Grossmann}}, \
  and\ \bibinfo {author} {\bibfnamefont {J.}~\bibnamefont {Zak}},\ }\href@noop
  {} {\bibfield  {journal} {\bibinfo  {journal} {Phys. Rev. B}\ }\textbf
  {\bibinfo {volume} {12}},\ \bibinfo {pages} {1118} (\bibinfo {year}
  {1975})}\BibitemShut {NoStop}%
\bibitem [{\citenamefont {Balian}(1981)}]{Bal}%
  \BibitemOpen
  \bibfield  {author} {\bibinfo {author} {\bibfnamefont {R.}~\bibnamefont
  {Balian}},\ }\href@noop {} {\bibfield  {journal} {\bibinfo  {journal} {C. R.
  Acad. Sci. Paris, Ser. 2}\ }\textbf {\bibinfo {volume} {292}},\ \bibinfo
  {pages} {1357} (\bibinfo {year} {1981})}\BibitemShut {NoStop}%
\bibitem [{\citenamefont {Ozyesil}\ \emph {et~al.}(2016)\citenamefont
  {Ozyesil}, \citenamefont {Sharon},\ and\ \citenamefont
  {Singer}}]{Sharon2016}%
  \BibitemOpen
  \bibfield  {author} {\bibinfo {author} {\bibfnamefont {O.}~\bibnamefont
  {Ozyesil}}, \bibinfo {author} {\bibfnamefont {N.}~\bibnamefont {Sharon}}, \
  and\ \bibinfo {author} {\bibfnamefont {A.}~\bibnamefont {Singer}},\
  }\href@noop {} {\enquote {\bibinfo {title} {Synchronization over cartan
  motion groups via contraction},}\ }\bibinfo {howpublished} {Preprint}
  (\bibinfo {year} {2016}),\ \bibinfo {note}
  {\href{https://arxiv.org/abs/1612.00059}{arXiv:1612.00059}}\BibitemShut
  {NoStop}%
\bibitem [{\citenamefont {Coifman}\ and\ \citenamefont
  {Wickerhause}(1992)}]{Coif}%
  \BibitemOpen
  \bibfield  {author} {\bibinfo {author} {\bibfnamefont {R.~R.}\ \bibnamefont
  {Coifman}}\ and\ \bibinfo {author} {\bibfnamefont {M.~V.}\ \bibnamefont
  {Wickerhause}},\ }\href@noop {} {\bibfield  {journal} {\bibinfo  {journal}
  {IEEE Trans. Inform. Theory}\ }\textbf {\bibinfo {volume} {38}},\ \bibinfo
  {pages} {713} (\bibinfo {year} {1992})}\BibitemShut {NoStop}%
\end{thebibliography}%

\end{document}